\documentclass[onefignum,onetabnum]{siamart220329}
\usepackage{bm}
\usepackage{algorithmic}
\usepackage{algorithm}
\usepackage{mathrsfs}
\usepackage{epsfig}
\usepackage{enumitem}
\usepackage{diagbox}
\usepackage{cases}
\usepackage{mathtools} 
\usepackage{tikz} 

\usepackage{stmaryrd}

\usepackage{xparse}

\NewDocumentCommand{\dgal}{sO{}m}{%
  \IfBooleanTF{#1}
    {\dgalext{#3}}
    {\dgalx[#2]{#3}}%
}

\NewDocumentCommand{\dgalext}{m}{%
  \sbox0{%
    \mathsurround=0pt 
    $\left\{\vphantom{#1}\right.\kern-\nulldelimiterspace$%
  }%
  \sbox2{\{}%
  \ifdim\ht0=\ht2
    \{\kern-.45\wd2 \{#1\}\kern-.45\wd2 \}%
  \else
  \fi
}

\NewDocumentCommand{\dgalx}{om}{%
  \sbox0{\mathsurround=0pt$#1\{$}%
  \sbox2{\{}%
  \ifdim\ht0=\ht2
    \{\kern-.45\wd2 \{#2\}\kern-.45\wd2 \}%
  \else
    \mathopen{#1\{\kern-.5\wd0 #1\{}
    #2
    \mathclose{#1\}\kern-.5\wd0 #1\}}
  \fi
}

\newcommand{\vertiii}[1]{{\left\vert\kern-0.25ex\left\vert\kern-0.25ex\left\vert #1 
\right\vert\kern-0.25ex\right\vert\kern-0.25ex\right\vert}}

\usepackage{lipsum}
\usepackage{amsfonts}
\usepackage{graphicx}
\usepackage{epstopdf}
\usepackage{algorithmic}

\usepackage{animate}
\usepackage{amsmath}
\usepackage{amssymb}
\usepackage{graphics}
\usepackage{graphicx}
\usepackage{footnote}
\usepackage{textcomp}
\usepackage{mathrsfs}
\usepackage{epstopdf}
\usepackage{array}
\usepackage[maxfloats=99]{morefloats}
\usepackage{url}
\usepackage{cases}
\usepackage{mathscinet}
\usepackage[normalem]{ulem}
\usepackage{algorithmic, algorithm}
\usepackage{color}
\usepackage{cite}

\ifpdf
  \DeclareGraphicsExtensions{.eps,.pdf,.png,.jpg}
\else
  \DeclareGraphicsExtensions{.eps}
\fi

\newsiamremark{remark}{Remark}
\newsiamremark{hypothesis}{Hypothesis}
\crefname{hypothesis}{Hypothesis}{Hypotheses}
\newsiamthm{claim}{Claim}

\headers{unfitted spectral element methods}{}

\title{Unfitted Spectral Element Method for interfacial models}

\author{Nicolas Gonzalez\thanks{Department of Mathematics, University of California, Santa Barbara, CA, 93106, USA(\email{nicogonzalez@math.ucsb.edu}).}\and 
Hailong Guo\thanks{ School of Mathematics and Statistics,  The University of Melbourne,  Parkville, VIC 3010, Australia   
  (\email{hailong.guo@unimelb.edu.au}).}
  \and Xu Yang\thanks{Department of Mathematics, University of California, Santa Barbara, CA, 93106, USA(\email{xuyang@math.ucsb.edu}).}
}

\usepackage{amsopn}

\usepackage{subcaption}

\usepackage{multirow}

\newtheorem{assumption}[theorem]{Assumption}
\graphicspath{{fig/}{fig/Eigen/}{fig/Poisson/}}

\usepackage{stmaryrd}

\begin{document}

\maketitle

\begin{abstract}
In this paper, we propose an unfitted spectral element method for solving elliptic interface and corresponding eigenvalue problems. The novelty of the proposed method lies in its combination of the spectral accuracy of the spectral element method and the flexibility of the unfitted Nitsche's method. We also use tailored ghost penalty terms to enhance its robustness. We establish optimal $hp$ convergence rates for both elliptic interface problems and interface eigenvalue problems. Additionally, we demonstrate spectral accuracy for model problems in terms of polynomial degree.

\end{abstract}



\section{Introduction}
\label{sec:int}

Interface problems arise naturally in various physical systems characterized by different background materials, with extensive applications in diverse fields such as fluid mechanics and materials science. The primary challenge for interface problems is the low regularity of the solution across the interface.
In the pioneering investigation of the finite element method for interface problems \cite{Babuska1970}, Babu\v{s}ka established that standard finite element methods can only achieve $\mathcal{O}(h^{1/2})$ accuracy unless the meshes conform precisely to the interface. To date, various methods have been developed to address interface problems. The existing numerical methods can be roughly categorized into two different classes: body-fitted mesh methods and unfitted  methods. When the meshes are fitted to the interface, optimal a priori error estimates are established \cite{ChenZou1998, Xu1982, Babuska1970}. To alleviate the need for body-fitted mesh generation for geometrically complicated interfaces, various unfitted numerical methods have been developed since the seminal work of Peskin on the immersed boundary method \cite{Peskin1977}. Famous examples include immersed interface methods \cite{LL1994}, immersed finite element methods \cite{Li1998a, LiLinWu2003}, ghost fluid method \cite{Liu2000}, Petrov-Galerkin methods \cite{HouLiu2005, HouWuZhang2004}, generalized/extended finite element methods \cite{Moes2001, BaOs1983}, and cut finite element methods \cite{Hansbo2002, Burman2015}.

The cut finite element method (CutFEM), also known as the unfitted Nitsche's method, was initially introduced by Hansbo et al. \cite{Hansbo2002}. The key idea behind this approach involves employing two distinct sets of basis functions on the interface elements. These sets of basis functions are weakly coupled using Nitsche's methods. Notably, this idea has been generalized to address various model equations, including elastic interface problems \cite{Hansbo2004}, Stokes interface problems \cite{Hansbo2014}, Maxwell interface problems \cite{LiuZhang2Zheng2020}, and biharmonic interface problems \cite{CaiChenWang2021}.
In our recent research \cite{GY20181}, we have established superconvergence results for the cut finite element method. Moreover, high-order analogues of this method have been developed \cite{ChenLiXiang2021, WuXiao2019, Lehr2016, GMassing2019, BadiaNe2022b}. To enhance the robustness of the method in the presence of arbitrarily small intersections between geometric and numerical meshes, innovative techniques such as the ghost penalty method \cite{Burman2012, Burman2010} and the cell aggregation technique \cite{BadiaNe2022, BadiaVer2018} have been proposed.
For readers interested in a comprehensive overview of CutFEM, we refer them to the review provided in \cite{Burman2015}.

The motivation for our paper stems from our recent investigation of unfitted finite element methods for interface eigenvalue problems \cite{GuoYangZ2021, GuoYangZhu2021}. These types of interface eigenvalue problems have important applications in materials sciences, particularly in band gap computations for photonic/phononic crystals and in edge model computations for topological materials. In the context of eigenvalue problems, as elucidated in \cite{Zhang2015}, higher-order numerical methods, especially spectral methods/spectral element methods, offer significantly more reliable numerical eigenvalues.

Our paper aims to introduce a novel and robust unfitted spectral element method for solving elliptic interface problems and associated interface eigenvalue problems. Unlike previous methodologies, we emphasize the utilization of nodal basis functions derived from the Legendre-Gauss-Lobatto points \cite{ShenTangWang2011} and the development of $hp$ error estimates. In pursuit of enhanced robustness, we incorporate a ghost penalty stabilization term with parameters tailored to the polynomial order $p$. Notably, for interface eigenvalue problems, both mass and stiffness terms require the inclusion of the ghost penalty. However, introducing an additional ghost penalty term in the mass term precludes the direct application of the Babu\v{s}ka-Osborne theory. To overcome this challenge, we propose a solution by introducing intermediate interface eigenvalue problems and their corresponding solution operators. Using the intermediate solution operator as a bridge, we decompose the eigenvalue approximation into two components: one that can be rigorously analyzed using the Babu\v{s}ka-Osborne theory and another that can be estimated through the operator super-approximation property.

The rest of the paper is organized as follows: In Section \ref{sec:model}, we introduce the equations for our interface models. Section \ref{sec:usem} presents the unfitted spectral element formulations of these model equations and establishes their stability. In Section \ref{sec:error}, we conduct \emph{a priori} error estimates. Section \ref{sec:num} includes several numerical examples that are consistent with the theoretical results. Finally, we make conclusive remarks in Section \ref{sec:con}.

\section{Model interface problems} 
\label{sec:model}

For simplicity's sake, we assume $\Omega \subset \mathbb R^2$ is a rectangular domain. We adopt the standard Sobolev spaces notation as in \cite{Evans2008, Ciarlet2002, BrennerScott2008}. For any subset $D \subseteq \Omega$, let $H^k(D)$ denote the Sobolev space with norm $\|\cdot \|_{k, D}$ and seminorm $|\cdot |_{k, D}$. For a domain $D = D_+\cup D_-$ with $D_+\cap D_- = \emptyset$, let $H^{k}\left(D_+ \cup D_-\right)$ be the function space consisting of piecewise Sobolev functions $w$ such that $\left. w\right|_{D_+} \in H^{k}\left(D_+\right)$ and $\left.w\right|_{D_-} \in H^{k}\left(D_-\right)$, whose norm is defined as
\begin{equation}
	\|w\|_{k,  D_+ \cup D_-}=\left(\|w\|_{k,  D_+}^2+\|w\|_{k, D_-}^2\right)^{1 / 2},
\end{equation}
and seminorm is defined as
\begin{equation}
	|w|_{k,  D_+ \cup D_-}=\left(|w|_{k, D_+}^2+|w|_{k,  D_-}^2\right)^{1 / 2} .
\end{equation}

Suppose there is a smooth curve $\Gamma$ separating $\Omega$ into two disjoint parts: $\Omega_+$ and $\Omega_-$ so can be decomposed as $\Omega = \Omega_- \cup \Gamma \cup \Omega_+$. We shall consider the following elliptic interface problem:
\begin{subequations}\label{equ:bvp}
\begin{alignat}{2}
	-\nabla \cdot (\alpha \nabla u)  & = f, \quad  \text{in } \Omega_- \cup \Omega_+, \\
    u & = 0, \quad \text{on } \partial\Omega, \\
    \llbracket u \rrbracket & = \bar{p}, \quad  \text{on } \Gamma, \\
    \llbracket \alpha \partial_{\mathbf{n}} u \rrbracket & = \bar{q}, \quad  \text{on } \Gamma,
\end{alignat}
\end{subequations}
where $\partial_{\mathbf{n}} u = \nabla u \cdot \mathbf{n}$ with $\mathbf{n}$ being the unit outward normal vector of $\Gamma$ from $\Omega_-$ to $\Omega_+$ and $\llbracket v \rrbracket(x) = v_+(x) - v_-(x)$ with $v_{\pm} = v|_{\Omega_{\pm}}$ being the restriction of $v$ in subdomain $\Omega_{\pm}$. The diffusion coefficient is piecewise defined as 
\begin{equation}\label{equ:coeff}
\alpha =
\left\{
\begin{array}{cc}
	\alpha_- & \text{ in } \Omega_-, \\
	\alpha_+ & \text{ in } \Omega_+. \\
\end{array}\right.
\end{equation}
We assume $\alpha_{\pm} \ge \alpha_0$ for some given $\alpha_0 > 0$.
Define the bilinear form 
\begin{equation}\label{equ:weak}
	a(v, w) = \int_{\Omega} \alpha \nabla v \cdot \nabla w \, dx.
\end{equation}

The second model equation is described as the following interface eigenvalue problem: we seek to find the eigenpair $(\lambda, u) \in (\mathbb{R}^+, H^1(\Omega))$ such that
\begin{subequations}\label{equ:eigenprob}
\begin{alignat}{2}
	-\nabla \cdot (\alpha \nabla u)  & = \lambda u, \; \text{in } \Omega_- \cup \Omega_+, \\
    u & = 0, \quad \text{on } \partial\Omega, \\
    \llbracket u \rrbracket & = 0, \quad \text{on } \Gamma, \\
    \llbracket \alpha \partial_{\mathbf{n}} u \rrbracket & = 0, \quad \text{on } \Gamma.
\end{alignat}
\end{subequations}
According to spectral theory, the interface eigenvalue problem \eqref{equ:eigenprob} possesses a countable sequence of real eigenvalues $0 < \lambda_1 \le \lambda_2 \cdots \rightarrow \infty$, along with corresponding eigenfunctions $u_1, u_2, \cdots$. These eigenfunctions can be assumed to satisfy $a(u_i, u_j) = \lambda_i(u_i, u_j) = \delta_{ij}$.

In this paper, the symbol $\mathcal{C}$, with or without a subscript, represents a generic constant, independent of mesh step size $h$, polynomial degree $p$, and interface $\Gamma$'s location. Its value may differ between instances. We shall write $x \lesssim y$ to denote the expression $x \leq \mathcal{C} y$.

\section{Unfitted spectral element method}\label{sec:usem}
In this section, we shall introduce the unfitted spectral element formulation for our model equations.

\subsection{Formulation of unfitted spectral element method}
\label{ssec:usem}
The spectral element method combines the flexibility of finite element methods with spectral accuracy, using orthogonal polynomials \cite{CHQZ2006, ShenTangWang2011}.

We denote the Legendre polynomials on the interval $[-1, 1]$ as $L_j(x)$ for $j \in \mathbb N_0$. The Lobatto polynomials are defined as
\begin{equation} \label{equ:lobatto}
	\Phi_{j+1} = \sqrt{\frac{2j+1}{2}}\int_{-1}^x L_j(t) \, dt, \quad j \ge 1. 
\end{equation}
The $(j+1)$ zeros of $\Phi_{j+1}$ are referred to as the Legendre-Gauss-Lobatto (LGL) points, denoted $\xi_0 = -1, \xi_1, \ldots, \xi_{j-1}, \xi_j = 1$. In 2D, the LGL points are the tensor product of their 1D counterpart. It's important to note that LGL nodes are non-uniformly distributed, hence minimizing Runge's phenomenon in computations.

Let $\hat{K} = [-1, 1] \times [-1, 1]$ be the reference element in 2D. We naturally fit the Lagrange interpolation polynomial of order $p$ on $\hat{K}$ in each dimension independently as follows. For some $i,j \in \{0, 1, \hdots, p\}$, 
\begin{equation} \label{equ:interp}
	\hat{\phi}_{i, j}(x, y) =  \left( \prod_{n = 0, n \neq i}^p\frac{x-\xi_n}{\xi_i - \xi_n} \right)
	 \left( \prod_{m = 0, m \neq j}^p\frac{y-\xi_m}{\xi_j - \xi_m} \right),
\end{equation}
The local spectral polynomial basis of degree $p$ on the reference element is then defined as
\begin{equation} \label{equ:localsemspace}
	\mathbb{Q}_{p}(\hat{K}) = \text{span}\{\hat{\phi}_{i, j}: 0 \le i, j \le p\}. 
\end{equation}
We can extend the above to any general rectangular element $K$, through an affine mapping $F_K:K \to \hat{K}$ i.e.
\begin{equation} \label{equ:gensemspace}
	\mathbb{Q}_{p}(K) = \text{span}\{\phi \circ F_K : \phi \in \mathbb{Q}_{p}(\hat{K}) \}.
\end{equation}

Let $\mathcal{T}_h = \{ K\}$ be a quasi-uniform rectangular partition of $\Omega$.  For any rectangle $K \in \mathcal{T}_h$, let $h_K$ represent the maximum length of $K$'s boundary edges, $h_K \coloneqq \max_{e \in \partial K} |e|$. Consequently, we define the mesh size $h$ to be the maximum of $h_K$ for all $K \in \mathcal{T}_h$, $h \coloneqq \max_{K \in \mathcal T_h} h_K$. An additional crucial aspect of the unfitted spectral element method involves categorizing mesh elements into two classes: those intersecting the interface, denoted by
\begin{equation} \label{equ:interfelems} \mathcal{T}_{\Gamma, h} \coloneqq \{K \in \mathcal{T}_h : K \cap \Gamma \neq \emptyset\}, \end{equation}
and those that do not. For any $K \in \mathcal{T}_{\Gamma,h}$, $K_{\pm} \coloneqq K \cap \Omega_{\pm}$, and $\Gamma_{K} \coloneqq \Gamma \cap K$. For $h$ sufficiently small, it is reasonable to assume $\mathcal{T}_h$ satisfies the following:
\begin{assumption}
Let $K$ be an interface element with boundary $\partial K$. $\Gamma$ intersects $\partial K$ exactly twice, and each open edge $e \in \partial K$ at most once. 
\label{ass:InterfIntersects}
\end{assumption}

Furthermore, the set of elements associated with $\Omega_{\pm}$ are
\begin{equation} \label{equ:submesh}
	\mathcal{T}_{\pm, h} \coloneqq \{ K \in \mathcal{T}_h: K \cap \Omega_\pm  \neq \emptyset \}.
\end{equation}
The fictitious domains $\Omega_{\pm, h} \supseteq \Omega_{\pm}$, are defined as
\begin{equation} \label{equ:fictdomain}
	\Omega_{\pm, h} \coloneqq \bigcup_{K \in \mathcal{T}_{\pm, h}} K. 
\end{equation}
An additional set of element faces is defined here to later on facilitate the ghost penalty bilinear form introduction. The set is 
\begin{equation} \label{equ:ghostface}
	\mathcal{G}_{\pm, h} = \left\{ \overline{K} \cap \overline{K'}:  K\neq K', K \in \mathcal{T}_{\Gamma, h} \text{ and }
	K' \in \mathcal{T}_{\pm, h} \right\}. 
\end{equation}

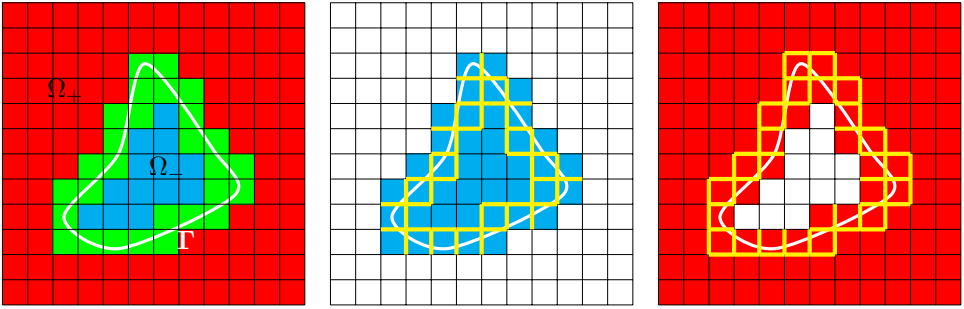
\begin{figure}[H]
    \centering
    \begin{minipage}{0.33\textwidth}
        \centering
        \begin{tikzpicture}[scale=0.67]
            \fill[red] (0,0) rectangle (0.5,0.5);
            \fill[red] (0.5,0) rectangle (1,0.5);
            \fill[red] (1,0) rectangle (1.5,0.5);
            \fill[red] (1.5,0) rectangle (2,0.5);
            \fill[red] (2,0) rectangle (2.5,0.5);
            \fill[red] (2.5,0) rectangle (3,0.5);
            \fill[red] (3,0) rectangle (3.5,0.5);
            \fill[red] (3.5,0) rectangle (4,0.5);
            \fill[red] (4,0) rectangle (4.5,0.5);
            \fill[red] (4.5,0) rectangle (5,0.5);
            \fill[red] (5,0) rectangle (5.5,0.5);
            \fill[red] (5.5,0) rectangle (6,0.5);
    
            \fill[red] (0,0.5) rectangle (0.5,1);
            \fill[red] (0.5,0.5) rectangle (1,1);
            \fill[red] (1,0.5) rectangle (1.5,1);
            \fill[red] (1.5,0.5) rectangle (2,1);
            \fill[red] (2,0.5) rectangle (2.5,1);
            \fill[red] (2.5,0.5) rectangle (3,1);
            \fill[red] (3,0.5) rectangle (3.5,1);
            \fill[red] (3.5,0.5) rectangle (4,1);
            \fill[red] (4,0.5) rectangle (4.5,1);
            \fill[red] (4.5,0.5) rectangle (5,1);
            \fill[red] (5,0.5) rectangle (5.5,1);
            \fill[red] (5.5,0.5) rectangle (6,1);
    
            \fill[red] (0,1) rectangle (0.5,1.5);
            \fill[red] (0.5,1) rectangle (1,1.5);
            \fill[green] (1,1) rectangle (1.5,1.5);
            \fill[green] (1.5,1) rectangle (2,1.5);
            \fill[green] (2,1) rectangle (2.5,1.5);
            \fill[green] (2.5,1) rectangle (3,1.5);
            \fill[green] (3,1) rectangle (3.5,1.5);
            \fill[red] (3.5,1) rectangle (4,1.5);
            \fill[red] (4,1) rectangle (4.5,1.5);
            \fill[red] (4.5,1) rectangle (5,1.5);
            \fill[red] (5,1) rectangle (5.5,1.5);
            \fill[red] (5.5,1) rectangle (6,1.5);
    
            \fill[red] (0,1.5) rectangle (0.5,2);
            \fill[red] (0.5,1.5) rectangle (1,2);
            \fill[green] (1,1.5) rectangle (1.5,2);
            \fill[cyan] (1.5,1.5) rectangle (2,2);
            \fill[cyan] (2,1.5) rectangle (2.5,2);
            \fill[cyan] (2.5,1.5) rectangle (3,2);
            \fill[green] (3,1.5) rectangle (3.5,2);
            \fill[green] (3.5,1.5) rectangle (4,2);
            \fill[green] (4,1.5) rectangle (4.5,2);
            \fill[red] (4.5,1.5) rectangle (5,2);
            \fill[red] (5,1.5) rectangle (5.5,2);
            \fill[red] (5.5,1.5) rectangle (6,2);
    
            \fill[red] (0,2) rectangle (0.5,2.5);
            \fill[red] (0.5,2) rectangle (1,2.5);
            \fill[green] (1,2) rectangle (1.5,2.5);
            \fill[green] (1.5,2) rectangle (2,2.5);
            \fill[cyan] (2,2) rectangle (2.5,2.5);
            \fill[cyan] (2.5,2) rectangle (3,2.5);
            \fill[cyan] (3,2) rectangle (3.5,2.5);
            \fill[cyan] (3.5,2) rectangle (4,2.5);
            \fill[green] (4,2) rectangle (4.5,2.5);
            \fill[green] (4.5,2) rectangle (5,2.5);
            \fill[red] (5,2) rectangle (5.5,2.5);
            \fill[red] (5.5,2) rectangle (6,2.5);
    
            \fill[red] (0,2.5) rectangle (0.5,3);
            \fill[red] (0.5,2.5) rectangle (1,3);
            \fill[red] (1,2.5) rectangle (1.5,3);
            \fill[green] (1.5,2.5) rectangle (2,3);
            \fill[green] (2,2.5) rectangle (2.5,3);
            \fill[cyan] (2.5,2.5) rectangle (3,3);
            \fill[cyan] (3,2.5) rectangle (3.5,3);
            \fill[cyan] (3.5,2.5) rectangle (4,3);
            \fill[green] (4,2.5) rectangle (4.5,3);
            \fill[green] (4.5,2.5) rectangle (5,3);
            \fill[red] (5,2.5) rectangle (5.5,3);
            \fill[red] (5.5,2.5) rectangle (6,3);
    
            \fill[red] (0,3) rectangle (0.5,3.5);
            \fill[red] (0.5,3) rectangle (1,3.5);
            \fill[red] (1,3) rectangle (1.5,3.5);
            \fill[red] (1.5,3) rectangle (2,3.5);
            \fill[green] (2,3) rectangle (2.5,3.5);
            \fill[cyan] (2.5,3) rectangle (3,3.5);
            \fill[cyan] (3,3) rectangle (3.5,3.5);
            \fill[green] (3.5,3) rectangle (4,3.5);
            \fill[green] (4,3) rectangle (4.5,3.5);
            \fill[red] (4.5,3) rectangle (5,3.5);
            \fill[red] (5,3) rectangle (5.5,3.5);
            \fill[red] (5.5,3) rectangle (6,3.5);
    
            \fill[red] (0,3.5) rectangle (0.5,4);
            \fill[red] (0.5,3.5) rectangle (1,4);
            \fill[red] (1,3.5) rectangle (1.5,4);
            \fill[red] (1.5,3.5) rectangle (2,4);
            \fill[green] (2,3.5) rectangle (2.5,4);
            \fill[green] (2.5,3.5) rectangle (3,4);
            \fill[cyan] (3,3.5) rectangle (3.5,4);
            \fill[green] (3.5,3.5) rectangle (4,4);
            \fill[red] (4,3.5) rectangle (4.5,4);
            \fill[red] (4.5,3.5) rectangle (5,4);
            \fill[red] (5,3.5) rectangle (5.5,4);
            \fill[red] (5.5,3.5) rectangle (6,4);
    
            \fill[red] (0,4) rectangle (0.5,4.5);
            \fill[red] (0.5,4) rectangle (1,4.5);
            \fill[red] (1,4) rectangle (1.5,4.5);
            \fill[red] (1.5,4) rectangle (2,4.5);
            \fill[red] (2,4) rectangle (2.5,4.5);
            \fill[green] (2.5,4) rectangle (3,4.5);
            \fill[green] (3,4) rectangle (3.5,4.5);
            \fill[green] (3.5,4) rectangle (4,4.5);
            \fill[red] (4,4) rectangle (4.5,4.5);
            \fill[red] (4.5,4) rectangle (5,4.5);
            \fill[red] (5,4) rectangle (5.5,4.5);
            \fill[red] (5.5,4) rectangle (6,4.5);
    
            \fill[red] (0,4.5) rectangle (0.5,5);
            \fill[red] (0.5,4.5) rectangle (1,5);
            \fill[red] (1,4.5) rectangle (1.5,5);
            \fill[red] (1.5,4.5) rectangle (2,5);
            \fill[red] (2,4.5) rectangle (2.5,5);
            \fill[green] (2.5,4.5) rectangle (3,5);
            \fill[green] (3,4.5) rectangle (3.5,5);
            \fill[red] (3.5,4.5) rectangle (4,5);
            \fill[red] (4,4.5) rectangle (4.5,5);
            \fill[red] (4.5,4.5) rectangle (5,5);
            \fill[red] (5,4.5) rectangle (5.5,5);
            \fill[red] (5.5,4.5) rectangle (6,5);
    
            \fill[red] (0,5) rectangle (0.5,5.5);
            \fill[red] (0.5,5) rectangle (1,5.5);
            \fill[red] (1,5) rectangle (1.5,5.5);
            \fill[red] (1.5,5) rectangle (2,5.5);
            \fill[red] (2,5) rectangle (2.5,5.5);
            \fill[red] (2.5,5) rectangle (3,5.5);
            \fill[red] (3,5) rectangle (3.5,5.5);
            \fill[red] (3.5,5) rectangle (4,5.5);
            \fill[red] (4,5) rectangle (4.5,5.5);
            \fill[red] (4.5,5) rectangle (5,5.5);
            \fill[red] (5,5) rectangle (5.5,5.5);
            \fill[red] (5.5,5) rectangle (6,5.5);
    
            \fill[red] (0,5.5) rectangle (0.5,6);
            \fill[red] (0.5,5.5) rectangle (1,6);
            \fill[red] (1,5.5) rectangle (1.5,6);
            \fill[red] (1.5,5.5) rectangle (2,6);
            \fill[red] (2,5.5) rectangle (2.5,6);
            \fill[red] (2.5,5.5) rectangle (3,6);
            \fill[red] (3,5.5) rectangle (3.5,6);
            \fill[red] (3.5,5.5) rectangle (4,6);
            \fill[red] (4,5.5) rectangle (4.5,6);
            \fill[red] (4.5,5.5) rectangle (5,6);
            \fill[red] (5,5.5) rectangle (5.5,6);
            \fill[red] (5.5,5.5) rectangle (6,6);
            \draw[very thick, white] plot [smooth cycle] coordinates {(1.22,1.74) (2.33,1.12) (4.61,2.18) (4.19,3.07) (2.83,4.79) (2.3,3)};
            \draw[step=0.5cm,very thin] (0,0) grid (6,6);
            \node[text=white] at (3.65,1.3) {$\bm{\Gamma}$};
            \node[] at (3.27,2.7) {$\Omega_-$};
            \node[] at (1.25,4.25) {$\Omega_+$};
        \end{tikzpicture}
    \end{minipage}\hfill
    \begin{minipage}{0.33\textwidth}
        \centering
        \begin{tikzpicture}[scale=0.67]
            \fill[cyan] (1,1) rectangle (1.5,1.5);
            \fill[cyan] (1.5,1) rectangle (2,1.5);
            \fill[cyan] (2,1) rectangle (2.5,1.5);
            \fill[cyan] (2.5,1) rectangle (3,1.5);
            \fill[cyan] (3,1) rectangle (3.5,1.5);
            
            \fill[cyan] (1,1.5) rectangle (1.5,2);
            \fill[cyan] (1.5,1.5) rectangle (2,2);
            \fill[cyan] (2,1.5) rectangle (2.5,2);
            \fill[cyan] (2.5,1.5) rectangle (3,2);
            \fill[cyan] (3,1.5) rectangle (3.5,2);
            \fill[cyan] (3.5,1.5) rectangle (4,2);
            \fill[cyan] (4,1.5) rectangle (4.5,2);
    
            \fill[cyan] (1,2) rectangle (1.5,2.5);
            \fill[cyan] (1.5,2) rectangle (2,2.5);
            \fill[cyan] (2,2) rectangle (2.5,2.5);
            \fill[cyan] (2.5,2) rectangle (3,2.5);
            \fill[cyan] (3,2) rectangle (3.5,2.5);
            \fill[cyan] (3.5,2) rectangle (4,2.5);
            \fill[cyan] (4,2) rectangle (4.5,2.5);
            \fill[cyan] (4.5,2) rectangle (5,2.5);
            
            \fill[cyan] (1.5,2.5) rectangle (2,3);
            \fill[cyan] (2,2.5) rectangle (2.5,3);
            \fill[cyan] (2.5,2.5) rectangle (3,3);
            \fill[cyan] (3,2.5) rectangle (3.5,3);
            \fill[cyan] (3.5,2.5) rectangle (4,3);
            \fill[cyan] (4,2.5) rectangle (4.5,3);
            \fill[cyan] (4.5,2.5) rectangle (5,3);
            
            \fill[cyan] (2,3) rectangle (2.5,3.5);
            \fill[cyan] (2.5,3) rectangle (3,3.5);
            \fill[cyan] (3,3) rectangle (3.5,3.5);
            \fill[cyan] (3.5,3) rectangle (4,3.5);
            \fill[cyan] (4,3) rectangle (4.5,3.5);
            
            \fill[cyan] (2,3.5) rectangle (2.5,4);
            \fill[cyan] (2.5,3.5) rectangle (3,4);
            \fill[cyan] (3,3.5) rectangle (3.5,4);
            \fill[cyan] (3.5,3.5) rectangle (4,4);
            
            \fill[cyan] (2.5,4) rectangle (3,4.5);
            \fill[cyan] (3,4) rectangle (3.5,4.5);
            \fill[cyan] (3.5,4) rectangle (4,4.5);
            
            \fill[cyan] (2.5,4.5) rectangle (3,5);
            \fill[cyan] (3,4.5) rectangle (3.5,5);
            
            \draw[very thick, white] plot [smooth cycle] coordinates {(1.22,1.74) (2.33,1.12) (4.61,2.18) (4.19,3.07) (2.83,4.79) (2.3,3)};
            \draw[step=0.5cm,very thin] (0,0) grid (6,6);
            \draw[ultra thick, yellow] (1.5,1) -- (1.5,1.5);
            \draw[ultra thick, yellow] (2,1) -- (2,1.5);
            \draw[ultra thick, yellow] (2.5,1) -- (2.5,1.5);
            \draw[ultra thick, yellow] (3,1) -- (3,1.5);
    
            \draw[ultra thick, yellow] (1.5,1.5) -- (1.5,2);
            \draw[ultra thick, yellow] (3,1.5) -- (3,2);
            \draw[ultra thick, yellow] (3.5,1.5) -- (3.5,2);
            \draw[ultra thick, yellow] (4,1.5) -- (4,2);
    
            \draw[ultra thick, yellow] (1.5,2) -- (1.5,2.5);
            \draw[ultra thick, yellow] (2,2) -- (2,2.5);
            \draw[ultra thick, yellow] (4,2) -- (4,2.5);
            \draw[ultra thick, yellow] (4.5,2) -- (4.5,2.5);
    
            \draw[ultra thick, yellow] (2,2.5) -- (2,3);
            \draw[ultra thick, yellow] (2.5,2.5) -- (2.5,3);
            \draw[ultra thick, yellow] (4,2.5) -- (4,3);
            \draw[ultra thick, yellow] (4.5,2.5) -- (4.5,3);
    
            \draw[ultra thick, yellow] (2.5,3) -- (2.5,3.5);
            \draw[ultra thick, yellow] (3.5,3) -- (3.5,3.5);
            \draw[ultra thick, yellow] (4,3) -- (4,3.5);
    
            \draw[ultra thick, yellow] (2.5,3.5) -- (2.5,4);
            \draw[ultra thick, yellow] (3,3.5) -- (3,4);
            \draw[ultra thick, yellow] (3.5,3.5) -- (3.5,4);
    
            \draw[ultra thick, yellow] (3,4) -- (3,4.5);
            \draw[ultra thick, yellow] (3.5,4) -- (3.5,4.5);
    
            \draw[ultra thick, yellow] (3,4.5) -- (3,5);
            \draw[ultra thick, yellow] (1,1.5) -- (1.5,1.5);
            \draw[ultra thick, yellow] (1.5,1.5) -- (2,1.5);
            \draw[ultra thick, yellow] (2,1.5) -- (2.5,1.5);
            \draw[ultra thick, yellow] (2.5,1.5) -- (3,1.5);
            \draw[ultra thick, yellow] (3,1.5) -- (3.5,1.5);
            
            \draw[ultra thick, yellow] (1,2) -- (1.5,2);
            \draw[ultra thick, yellow] (1.5,2) -- (2,2);
            \draw[ultra thick, yellow] (3,2) -- (3.5,2);
            \draw[ultra thick, yellow] (3.5,2) -- (4,2);
            \draw[ultra thick, yellow] (4,2) -- (4.5,2);
    
            \draw[ultra thick, yellow] (1.5,2.5) -- (2,2.5);
            \draw[ultra thick, yellow] (2,2.5) -- (2.5,2.5);
            \draw[ultra thick, yellow] (4,2.5) -- (4.5,2.5);
            \draw[ultra thick, yellow] (4.5,2.5) -- (5,2.5);
    
            \draw[ultra thick, yellow] (2,3) -- (2.5,3);
            \draw[ultra thick, yellow] (3.5,3) -- (4,3);
            \draw[ultra thick, yellow] (4,3) -- (4.5,3);
    
            \draw[ultra thick, yellow] (2,3.5) -- (2.5,3.5);
            \draw[ultra thick, yellow] (2.5,3.5) -- (3,3.5);
            \draw[ultra thick, yellow] (3.5,3.5) -- (4,3.5);
            
            \draw[ultra thick, yellow] (2.5,4) -- (3,4);
            \draw[ultra thick, yellow] (3,4) -- (3.5,4);
            \draw[ultra thick, yellow] (3.5,4) -- (4,4);
    
            \draw[ultra thick, yellow] (2.5,4.5) -- (3,4.5);
            \draw[ultra thick, yellow] (3,4.5) -- (3.5,4.5);
        \end{tikzpicture}
    \end{minipage}\hfill
    \begin{minipage}{0.33\textwidth}
        \centering
        \begin{tikzpicture}[scale=0.67]
            \fill[red] (0,0) rectangle (0.5,0.5);
            \fill[red] (0.5,0) rectangle (1,0.5);
            \fill[red] (1,0) rectangle (1.5,0.5);
            \fill[red] (1.5,0) rectangle (2,0.5);
            \fill[red] (2,0) rectangle (2.5,0.5);
            \fill[red] (2.5,0) rectangle (3,0.5);
            \fill[red] (3,0) rectangle (3.5,0.5);
            \fill[red] (3.5,0) rectangle (4,0.5);
            \fill[red] (4,0) rectangle (4.5,0.5);
            \fill[red] (4.5,0) rectangle (5,0.5);
            \fill[red] (5,0) rectangle (5.5,0.5);
            \fill[red] (5.5,0) rectangle (6,0.5);
    
            \fill[red] (0,0.5) rectangle (0.5,1);
            \fill[red] (0.5,0.5) rectangle (1,1);
            \fill[red] (1,0.5) rectangle (1.5,1);
            \fill[red] (1.5,0.5) rectangle (2,1);
            \fill[red] (2,0.5) rectangle (2.5,1);
            \fill[red] (2.5,0.5) rectangle (3,1);
            \fill[red] (3,0.5) rectangle (3.5,1);
            \fill[red] (3.5,0.5) rectangle (4,1);
            \fill[red] (4,0.5) rectangle (4.5,1);
            \fill[red] (4.5,0.5) rectangle (5,1);
            \fill[red] (5,0.5) rectangle (5.5,1);
            \fill[red] (5.5,0.5) rectangle (6,1);
    
            \fill[red] (0,1) rectangle (0.5,1.5);
            \fill[red] (0.5,1) rectangle (1,1.5);
            \fill[red] (1,1) rectangle (1.5,1.5);
            \fill[red] (1.5,1) rectangle (2,1.5);
            \fill[red] (2,1) rectangle (2.5,1.5);
            \fill[red] (2.5,1) rectangle (3,1.5);
            \fill[red] (3,1) rectangle (3.5,1.5);
            \fill[red] (3.5,1) rectangle (4,1.5);
            \fill[red] (4,1) rectangle (4.5,1.5);
            \fill[red] (4.5,1) rectangle (5,1.5);
            \fill[red] (5,1) rectangle (5.5,1.5);
            \fill[red] (5.5,1) rectangle (6,1.5);
    
            \fill[red] (0,1.5) rectangle (0.5,2);
            \fill[red] (0.5,1.5) rectangle (1,2);
            \fill[red] (1,1.5) rectangle (1.5,2);
            \fill[red] (3,1.5) rectangle (3.5,2);
            \fill[red] (3.5,1.5) rectangle (4,2);
            \fill[red] (4,1.5) rectangle (4.5,2);
            \fill[red] (4.5,1.5) rectangle (5,2);
            \fill[red] (5,1.5) rectangle (5.5,2);
            \fill[red] (5.5,1.5) rectangle (6,2);
    
            \fill[red] (0,2) rectangle (0.5,2.5);
            \fill[red] (0.5,2) rectangle (1,2.5);
            \fill[red] (1,2) rectangle (1.5,2.5);
            \fill[red] (1.5,2) rectangle (2,2.5);
            \fill[red] (4,2) rectangle (4.5,2.5);
            \fill[red] (4.5,2) rectangle (5,2.5);
            \fill[red] (5,2) rectangle (5.5,2.5);
            \fill[red] (5.5,2) rectangle (6,2.5);
    
            \fill[red] (0,2.5) rectangle (0.5,3);
            \fill[red] (0.5,2.5) rectangle (1,3);
            \fill[red] (1,2.5) rectangle (1.5,3);
            \fill[red] (1.5,2.5) rectangle (2,3);
            \fill[red] (2,2.5) rectangle (2.5,3);
            \fill[red] (4,2.5) rectangle (4.5,3);
            \fill[red] (4.5,2.5) rectangle (5,3);
            \fill[red] (5,2.5) rectangle (5.5,3);
            \fill[red] (5.5,2.5) rectangle (6,3);
    
            \fill[red] (0,3) rectangle (0.5,3.5);
            \fill[red] (0.5,3) rectangle (1,3.5);
            \fill[red] (1,3) rectangle (1.5,3.5);
            \fill[red] (1.5,3) rectangle (2,3.5);
            \fill[red] (2,3) rectangle (2.5,3.5);
            \fill[red] (3.5,3) rectangle (4,3.5);
            \fill[red] (4,3) rectangle (4.5,3.5);
            \fill[red] (4.5,3) rectangle (5,3.5);
            \fill[red] (5,3) rectangle (5.5,3.5);
            \fill[red] (5.5,3) rectangle (6,3.5);
    
            \fill[red] (0,3.5) rectangle (0.5,4);
            \fill[red] (0.5,3.5) rectangle (1,4);
            \fill[red] (1,3.5) rectangle (1.5,4);
            \fill[red] (1.5,3.5) rectangle (2,4);
            \fill[red] (2,3.5) rectangle (2.5,4);
            \fill[red] (2.5,3.5) rectangle (3,4);
            \fill[red] (3.5,3.5) rectangle (4,4);
            \fill[red] (4,3.5) rectangle (4.5,4);
            \fill[red] (4.5,3.5) rectangle (5,4);
            \fill[red] (5,3.5) rectangle (5.5,4);
            \fill[red] (5.5,3.5) rectangle (6,4);
    
            \fill[red] (0,4) rectangle (0.5,4.5);
            \fill[red] (0.5,4) rectangle (1,4.5);
            \fill[red] (1,4) rectangle (1.5,4.5);
            \fill[red] (1.5,4) rectangle (2,4.5);
            \fill[red] (2,4) rectangle (2.5,4.5);
            \fill[red] (2.5,4) rectangle (3,4.5);
            \fill[red] (3,4) rectangle (3.5,4.5);
            \fill[red] (3.5,4) rectangle (4,4.5);
            \fill[red] (4,4) rectangle (4.5,4.5);
            \fill[red] (4.5,4) rectangle (5,4.5);
            \fill[red] (5,4) rectangle (5.5,4.5);
            \fill[red] (5.5,4) rectangle (6,4.5);
    
            \fill[red] (0,4.5) rectangle (0.5,5);
            \fill[red] (0.5,4.5) rectangle (1,5);
            \fill[red] (1,4.5) rectangle (1.5,5);
            \fill[red] (1.5,4.5) rectangle (2,5);
            \fill[red] (2,4.5) rectangle (2.5,5);
            \fill[red] (2.5,4.5) rectangle (3,5);
            \fill[red] (3,4.5) rectangle (3.5,5);
            \fill[red] (3.5,4.5) rectangle (4,5);
            \fill[red] (4,4.5) rectangle (4.5,5);
            \fill[red] (4.5,4.5) rectangle (5,5);
            \fill[red] (5,4.5) rectangle (5.5,5);
            \fill[red] (5.5,4.5) rectangle (6,5);
    
            \fill[red] (0,5) rectangle (0.5,5.5);
            \fill[red] (0.5,5) rectangle (1,5.5);
            \fill[red] (1,5) rectangle (1.5,5.5);
            \fill[red] (1.5,5) rectangle (2,5.5);
            \fill[red] (2,5) rectangle (2.5,5.5);
            \fill[red] (2.5,5) rectangle (3,5.5);
            \fill[red] (3,5) rectangle (3.5,5.5);
            \fill[red] (3.5,5) rectangle (4,5.5);
            \fill[red] (4,5) rectangle (4.5,5.5);
            \fill[red] (4.5,5) rectangle (5,5.5);
            \fill[red] (5,5) rectangle (5.5,5.5);
            \fill[red] (5.5,5) rectangle (6,5.5);
    
            \fill[red] (0,5.5) rectangle (0.5,6);
            \fill[red] (0.5,5.5) rectangle (1,6);
            \fill[red] (1,5.5) rectangle (1.5,6);
            \fill[red] (1.5,5.5) rectangle (2,6);
            \fill[red] (2,5.5) rectangle (2.5,6);
            \fill[red] (2.5,5.5) rectangle (3,6);
            \fill[red] (3,5.5) rectangle (3.5,6);
            \fill[red] (3.5,5.5) rectangle (4,6);
            \fill[red] (4,5.5) rectangle (4.5,6);
            \fill[red] (4.5,5.5) rectangle (5,6);
            \fill[red] (5,5.5) rectangle (5.5,6);
            \fill[red] (5.5,5.5) rectangle (6,6);
            \draw[very thick, white] plot [smooth cycle] coordinates {(1.22,1.74) (2.33,1.12) (4.61,2.18) (4.19,3.07) (2.83,4.79) (2.3,3)};
            \draw[step=0.5cm,very thin] (0,0) grid (6,6);
            \draw[ultra thick, yellow] (1,1) -- (1,1.5);
            \draw[ultra thick, yellow] (1.5,1) -- (1.5,1.5);
            \draw[ultra thick, yellow] (2,1) -- (2,1.5);
            \draw[ultra thick, yellow] (2.5,1) -- (2.5,1.5);
            \draw[ultra thick, yellow] (3,1) -- (3,1.5);
            \draw[ultra thick, yellow] (3.5,1) -- (3.5,1.5);
    
            \draw[ultra thick, yellow] (1,1.5) -- (1,2);
            \draw[ultra thick, yellow] (3.5,1.5) -- (3.5,2);
            \draw[ultra thick, yellow] (4,1.5) -- (4,2);
            \draw[ultra thick, yellow] (4.5,1.5) -- (4.5,2);
    
            \draw[ultra thick, yellow] (1,2) -- (1,2.5);
            \draw[ultra thick, yellow] (1.5,2) -- (1.5,2.5);
            \draw[ultra thick, yellow] (4.5,2) -- (4.5,2.5);
            \draw[ultra thick, yellow] (5,2) -- (5,2.5);
    
            \draw[ultra thick, yellow] (1.5,2.5) -- (1.5,3);
            \draw[ultra thick, yellow] (2,2.5) -- (2,3);
            \draw[ultra thick, yellow] (4.5,2.5) -- (4.5,3);
            \draw[ultra thick, yellow] (5,2.5) -- (5,3);
    
            \draw[ultra thick, yellow] (2,3) -- (2,3.5);
            \draw[ultra thick, yellow] (4,3) -- (4,3.5);
            \draw[ultra thick, yellow] (4.5,3) -- (4.5,3.5);
    
            \draw[ultra thick, yellow] (2,3.5) -- (2,4);
            \draw[ultra thick, yellow] (2.5,3.5) -- (2.5,4);
            \draw[ultra thick, yellow] (4,3.5) -- (4,4);
    
            \draw[ultra thick, yellow] (2.5,4) -- (2.5,4.5);
            \draw[ultra thick, yellow] (3,4) -- (3,4.5);
            \draw[ultra thick, yellow] (3.5,4) -- (3.5,4.5);
            \draw[ultra thick, yellow] (4,4) -- (4,4.5);
    
            \draw[ultra thick, yellow] (2.5,4.5) -- (2.5,5);
            \draw[ultra thick, yellow] (3,4.5) -- (3,5);
            \draw[ultra thick, yellow] (3.5,4.5) -- (3.5,5);
            \draw[ultra thick, yellow] (1,1) -- (1.5,1);
            \draw[ultra thick, yellow] (1.5,1) -- (2,1);
            \draw[ultra thick, yellow] (2,1) -- (2.5,1);
            \draw[ultra thick, yellow] (2.5,1) -- (3,1);
            \draw[ultra thick, yellow] (3,1) -- (3.5,1);
    
            \draw[ultra thick, yellow] (1,1.5) -- (1.5,1.5);
            \draw[ultra thick, yellow] (3,1.5) -- (3.5,1.5);
            \draw[ultra thick, yellow] (3.5,1.5) -- (4,1.5);
            \draw[ultra thick, yellow] (4,1.5) -- (4.5,1.5);
    
            \draw[ultra thick, yellow] (1,2) -- (1.5,2);
            \draw[ultra thick, yellow] (4,2) -- (4.5,2);
            \draw[ultra thick, yellow] (4.5,2) -- (5,2);
    
            \draw[ultra thick, yellow] (1,2.5) -- (1.5,2.5);
            \draw[ultra thick, yellow] (1.5,2.5) -- (2,2.5);
            \draw[ultra thick, yellow] (4,2.5) -- (4.5,2.5);
            \draw[ultra thick, yellow] (4.5,2.5) -- (5,2.5);
    
            \draw[ultra thick, yellow] (1.5,3) -- (2,3);
            \draw[ultra thick, yellow] (2,3) -- (2.5,3);
            \draw[ultra thick, yellow] (4,3) -- (4.5,3);
            \draw[ultra thick, yellow] (4.5,3) -- (5,3);
    
            \draw[ultra thick, yellow] (2,3.5) -- (2.5,3.5);
            \draw[ultra thick, yellow] (3.5,3.5) -- (4,3.5);
            \draw[ultra thick, yellow] (4,3.5) -- (4.5,3.5);
    
            \draw[ultra thick, yellow] (2,4) -- (2.5,4);
            \draw[ultra thick, yellow] (2.5,4) -- (3,4);
            \draw[ultra thick, yellow] (3.5,4) -- (4,4);
    
            \draw[ultra thick, yellow] (2.5,4.5) -- (3,4.5);
            \draw[ultra thick, yellow] (3,4.5) -- (3.5,4.5);
            \draw[ultra thick, yellow] (3.5,4.5) -- (4,4.5);
    
            \draw[ultra thick, yellow] (2.5,5) -- (3,5);
            \draw[ultra thick, yellow] (3,5) -- (3.5,5);
        \end{tikzpicture}
    \end{minipage}
    \caption{Left: $\Omega$ and interface $\Gamma$ (white curve described as a level-set). Green elements: $\Omega_{\Gamma,h}$. Blue elements: $\Omega_{-,h} \backslash \mathcal T_{\Gamma,h}$. Red elements: $\Omega_{+,h} \backslash \mathcal T_{\Gamma,h}$. Middle: Blue elements: $\Omega_{-,h}$. Yellow edges: $\mathcal G_{-,h}$. Right: Red elements: $\Omega_{+,h}$. Yellow edges: $\mathcal G_{+,h}$. Note that a ghost penalty face may belong to both $\Omega_-,\Omega_+$.}
\end{figure}

With the above laid out, we define the spectral element space on each subdomain $\Omega_{\pm, h}$,
\begin{equation} \label{equ:semspace}
    V_\pm^{p, h} \coloneqq \{ u^{p, h} : u^{p, h} \in C^0(\Omega_{\pm, h}), \; u^{p,h}|_K\in \mathbb{Q}_p(K), \; \forall \; K \in \mathcal{T}_{\pm, h} \}.
\end{equation}
The unfitted spectral element spaces are direct sums of the above, i.e.
\begin{equation} \label{equ:usemspace}
    V^{p, h} \coloneqq V_-^{p, h} \oplus V_+^{p, h}.
\end{equation}
 For any  interface element $K \in \mathcal T_{\Gamma,h}$, there are two independent sets of basis functions belonging to both $V_\pm^{p,h}$ respectively.
 To handle the homogeneous Dirichlet boundary condition, we introduce the subspace
\begin{equation}\label{equ:h0semspace}
	V^{p, h}_0 \coloneqq V^{p,h}\cap H^1_0(\Omega). 
\end{equation}

We present more notation before proceeding to the bilinear form. The stabilizing weights are (as in \cite{Dolbow2012}),
\begin{equation}\label{equ:weight}
	\kappa_+ = \frac{\alpha_-|K_+|}{\alpha_-|K_+|+\alpha_+|K_-|}
   \quad \text{ and }\quad
   \kappa_- = \frac{\alpha_+|K_-|}{\alpha_-|K_+|+\alpha_+|K_-|},
\end{equation}
where $|\cdot|$ represents the appropriate set measure. For $v^{p,h} = (v_-^{p,h}, v_+^{p, h}) \in V_-^{p, h} \oplus V_+^{p, h}$, the average operators along the interface are
\begin{equation}\label{equ:average}
	\dgal{v^{p,h}} \coloneqq \kappa_- v_-^{p,h} + \kappa_+ v_+^{p,h},
\quad 
	\dgal{v^{p,h}}^\ast \coloneqq \kappa_+ v_-^{p,h} + \kappa_- v_+^{p,h},
\end{equation}
and the jump, in the limit sense, is
\begin{equation}\label{equ:jumpgamma2}
	\llbracket v^{p,h}  \rrbracket \coloneqq v_+^{p,h} - v_-^{p,h}.
\end{equation}
Likewise, a jump on any edge $e \in \mathcal{G}_{\pm, h}$, can be described as
\begin{equation}\label{equ:jumpedge}
	[v^{p,h}](x) = \lim_{t\rightarrow 0^+}v^{p,h}(x+t\mathbf{n}) -  \lim_{t\rightarrow 0^+}v^{p,h}(x-t\mathbf{n}), \quad x\in e,
\end{equation}
with $\mathbf{n}$ being $e$'s unit normal vector with arbitrary fixed orientation.

With the above, we are now able to introduce our bilinear form
\begin{equation}\label{equ:bilinear}
	\begin{aligned}
	a_{p,h}(v^{p,h}, w^{p,h}) \coloneqq
		&\sum\limits_{i=\pm}\left(\alpha_i\nabla v_{i}^{p,h}, \nabla w_{i}^{p,h}\right)_{\Omega_i}
+ \left\langle  \llbracket v^{p,h} \rrbracket ,\dgal{\alpha \partial_{\mathbf{n}} w^{p,h}}\right\rangle_{\Gamma} \\
&+ \left\langle  \dgal{\alpha \partial_{\mathbf{n}} v^{p,h} }, \llbracket w^{p,h} \rrbracket \right\rangle_{\Gamma}
+ \frac{\gamma p^2}{h} \left\langle \llbracket v^{p,h} \rrbracket,  \llbracket w^{p,h} \rrbracket\right\rangle_{\Gamma},
	\end{aligned}
\end{equation}
where the penalty (stabilization) parameter $\gamma$ on the interface element $K$ is
\begin{equation}\label{equ:gamma}
	\gamma|_K = \frac{2h_K|\Gamma_K|}{|K_+|/\alpha_++|K_-|/\alpha_-}.
\end{equation}

To  enhance the robustness of the proposed unfitted spectral element method, we introduce the ghost penalty term 
\begin{equation}\label{equ:ghostpenalty}
	g_{p,h}(v^{p,h}, w^{p,h}) \coloneqq \sum_{e \in \mathcal{G}_{\pm,h}} \sum_{j=0}^{p} \frac{h^{2j+1}}{p^{2j}}([ \partial^j_{\mathbf{n}} v^{p,h} ], [ \partial^j_{\mathbf{n}} w^{p,h} ])_e. 
\end{equation}
\begin{remark}
    Only elements near vicinity of interface are involved in the ghost penalty bilinear form.
\end{remark}
\begin{remark}
In the coefficient of the ghost penalty term, we incorporate the dependence on the polynomial degree $p$ in the denominator. This change is necessary to ensure the accuracy of the error estimates derived below. We also observe numerical improvements with this alteration.
\end{remark}

The linear functional $l_h(\cdot)$ is defined as
\begin{equation}\label{equ:linearfun}
	l_{p,h}(v^{p,h}) \coloneqq \sum_{i=\pm} (f_i, v_i^{p, h})_{\Omega_i} + \left\langle  \bar{p} , \dgal{\alpha \partial_{\mathbf{n}} v^{p,h}} \right\rangle_{\Gamma} + \frac{\gamma p^2}{h} \left\langle \bar{p} , \llbracket v^{p,h} \rrbracket \right\rangle_{\Gamma} + \left\langle  \bar{q} , \dgal{ v^{p,h}}^* \right\rangle_{\Gamma}.
\end{equation}

The unfitted spectral element method for solving the interface problem defined by \eqref{equ:bvp} seeks to determine $u^{p,h} \in V_0^{p,h}$, such that the following equation holds for all $v^{p,h}\in V_0^{p,h}$:
\begin{equation}\label{equ:usem_bvp}
	A_{p,h}(u^{p,h}, v^{p,h}) 
	= l_{p,h}(v^{p,h}).
\end{equation}
where the bilinear form $A_{p,h}(\cdot, \cdot)$ is 
\begin{equation}\label{equ:extended_bilinear}
A_{p,h}(u^{p,h}, v^{p,h}) \coloneqq a_{p,h}(u^{p,h}, v^{p,h}) + \frac{\gamma_A}{h^2}g_{p,h}(u^{p,h}, v^{p,h}).
\end{equation}
with $\gamma_A$ being the ghost penalty parameter to be specified.

Similarly, the unfitted spectral element method applied to solving the interface eigenvalue problem described by equation \eqref{equ:eigenprob} aims to find $(\lambda^{p,h}, u^{p,h}) \in (\mathbb{R}^+, V_0^{p,h})$, such that the following equation holds for all $v^{p,h} \in V_0^{p,h}$:
\begin{equation}\label{equ:usem_eigen}
	A_{p,h}(u^{p,h}, v^{p,h})  = \lambda^{p,h}M_{p,h}(u^{p,h}, v^{p,h}),
\end{equation}
where the bilinear form  $M_{p,h}(u^{p,h}, v^{p,h})$ is defined as:
\begin{equation}\label{equ:extend_mass}
M_{p,h}(u^{p,h}, v^{p,h}) \coloneqq \sum_{i=\pm} (u_i^{p,h}, v_i^{p,h})_{\Omega_i} + \gamma_M g_{p,h}(u^{p,h}, v^{p,h}),
\end{equation}
with $\gamma_M$ being the ghost penalty parameter to be specified.

\begin{remark}
The ghost penalty term, ($\ref{equ:ghostpenalty}$), serves a crucial role in enhancing unfitted numerical methods' robustness, particularly in scenarios involving small cut geometries. For the unfitted spectral element method, we take into account its dependence on the polynomial degree $p$.
\end{remark}

\subsection{Stability analysis}
\label{ssec:stab}
In this subsection, we shall establish the well-posedness of the proposed unfitted spectral element methods. We commence by introducing the following definitions of energy norms:
\begin{align}
    &\vertiii{v}^2 \coloneqq \|\nabla v \|_{0, \Omega_- \cup \Omega_+}^2 + \| \dgal{\partial_{\mathbf{n}} v} \|_{-1,p,h,\Gamma}^2 + \| \llbracket v \rrbracket \|_{1,p, h,\Gamma}^2, \label{equ:energynorm}\\
    &\vertiii{v}_{\ast}^2 \coloneqq \vertiii{v}^2 + \frac{\gamma_A}{h^2} g_{p,h}(v, v), \label{equ:augmentedenergynorm}
\end{align}
where 
\begin{align}
	\| w \|_{-1,p,h,\Gamma}^2 \coloneqq \frac{h}{p^2} \sum_{K \in \mathcal{T}_{\Gamma,h}}\|w\|_{0, \Gamma_K}^2, 
		\text{ and } \| w \|_{1,p,h,\Gamma}^2 \coloneqq \frac{p^2}{h} \sum_{K \in \mathcal{T}_{\Gamma,h}}\|w\|_{0, \Gamma_K}^2.
\end{align}

To establish the well-posedness of discrete problems, it is imperative to introduce the following lemma:
\begin{theorem}\label{thm:wellposedness}
	For any $v, w \in V^{p,h}$, it holds that
	\begin{align}
			A_{p,h}(v, w) &\lesssim \vertiii{v}_{\ast}\vertiii{w}_{\ast}, \label{equ:continuity}\\
			\vertiii{v}_{\ast}^2 &\lesssim A_{p,h}(v, v) \label{equ:coercivity}.
	\end{align}
\end{theorem}
\begin{proof}
	The  continuity of bilinear form $A_{p,h}(\cdot, \cdot)$ can be proved using the Cauchy-Schwartz inequality. For the coercivity, it can be proved using the same approach as in \cite{GMassing2019} by using the enhanced stability of the ghost penalty term. 
\end{proof}

\begin{remark}
 We shall establish the continuity of the bilinear form $a_{p,h}(\cdot, \cdot)$ with respect to the energy norm $\vertiii{\cdot }$. In particular, we have\
 \begin{equation} \label{equ:newcont}
 	a_{p,h}(v, w)\lesssim \vertiii{v }\vertiii{w}, \quad \forall v,w\in V^{p,h}.
 \end{equation}
 \end{remark}

Combining Theorem \ref{thm:wellposedness} with the Lax-Milgram theorem \cite{Evans2008}, we can conclude that the discrete problem \eqref{equ:usem_bvp} admits a unique solution and the discrete eigenvalue problem \eqref{equ:usem_eigen} is well-posed. Furthermore, according to the spectral theory \cite{BaOs1991}, the discrete eigenvalues can be ordered as:
\begin{equation}
	0 < \lambda_1^{p,h} \le \lambda_2^{p,h} \le \cdots \le \lambda_{N^{p,h}}^{p,h},
\end{equation}
where $N^{p,h}$ represents the number of basis functions in $V_0^{p,h}$. The associated orthonormal eigenfunctions are denoted as $u_i^{p,h}$ for $i=1, \hdots, N^{p,h}$.

\subsection{High-order quadrature}
\label{ssec:quad}
In this subsection, we will elucidate the numerical integration methods employed to compute the integral quantities within both the bilinear and linear forms. Determining the weights and nodes for our quadrature is of paramount importance.

\begin{figure}[!h]
    \centering
    \begin{tikzpicture}[scale=0.6]
        \draw[very thick] (0,0) rectangle (12,8);
        \draw [very thick, green] plot [smooth] coordinates {(2,-0.7) (3,2) (4,5) (6,3) (8,6) (9,6.3) (10,8.5)};
        \node[text=green] at (10.2,8.8) {$\bm{\Gamma}$};
        \filldraw[brown] (0,0) circle (3pt) node[anchor=north]{$(r_0,\tilde{y}_m)$};
        \filldraw[brown] (2.3,0) circle (3pt) node[anchor=north west]{$(r_1,\tilde{y}_m)$};
        \filldraw[brown] (9.8,8) circle (3pt) node[anchor=south east]{$(r_2,\tilde{y}_M)$};
        \filldraw[brown] (12,8) circle (3pt) node[anchor=south]{$(r_3,\tilde{y}_M)$};
        \draw [dashed] (2.3,0) -- (2.3,8);
        \draw [dashed] (9.8,0) -- (9.8,8);
        \filldraw[red] (0.16,0.555) circle (3pt);
        \filldraw[red] (0.76,0.555) circle (3pt);
        \filldraw[red] (1.54,0.555) circle (3pt);
        \filldraw[red] (2.14,0.555) circle (3pt);

        \filldraw[red] (0.16,2.64) circle (3pt);
        \filldraw[red] (0.76,2.64) circle (3pt);
        \filldraw[red] (1.54,2.64) circle (3pt);
        \filldraw[red] (2.14,2.64) circle (3pt);

        \filldraw[red] (0.16,5.36) circle (3pt);
        \filldraw[red] (0.76,5.36) circle (3pt);
        \filldraw[red] (1.54,5.36) circle (3pt);
        \filldraw[red] (2.14,5.36) circle (3pt);

        \filldraw[red] (0.16,7.444) circle (3pt);
        \filldraw[red] (0.76,7.444) circle (3pt);
        \filldraw[red] (1.54,7.444) circle (3pt);
        \filldraw[red] (2.14,7.444) circle (3pt);

        \filldraw[red] (2.82,1.951) circle (3pt);
        \filldraw[red] (4.775,4.464) circle (3pt);
        \filldraw[red] (7.325,5.115) circle (3pt);
        \filldraw[red] (9.279,6.883) circle (3pt);

        \filldraw[red] (2.82,3.645) circle (3pt);
        \filldraw[red] (4.775,5.454) circle (3pt);
        \filldraw[red] (7.325,5.923) circle (3pt);
        \filldraw[red] (9.279,7.196) circle (3pt);

        \filldraw[red] (2.82,5.855) circle (3pt);
        \filldraw[red] (4.775,6.746) circle (3pt);
        \filldraw[red] (7.325,6.977) circle (3pt);
        \filldraw[red] (9.279,7.604) circle (3pt);

        \filldraw[red] (2.82,7.549) circle (3pt);
        \filldraw[red] (4.775,7.736) circle (3pt);
        \filldraw[red] (7.325,7.785) circle (3pt);
        \filldraw[red] (9.279,7.917) circle (3pt);
        
        \filldraw[blue] (2.82,0.104) circle (3pt);
        \filldraw[blue] (4.775,0.292) circle (3pt);
        \filldraw[blue] (7.325,0.34) circle (3pt);
        \filldraw[blue] (9.279,0.472) circle (3pt);

        \filldraw[blue] (2.82,0.495) circle (3pt);
        \filldraw[blue] (4.775,1.386) circle (3pt);
        \filldraw[blue] (7.325,1.617) circle (3pt);
        \filldraw[blue] (9.279,2.244) circle (3pt);

        \filldraw[blue] (2.82,1.005) circle (3pt);
        \filldraw[blue] (4.775,2.814) circle (3pt);
        \filldraw[blue] (7.325,3.283) circle (3pt);
        \filldraw[blue] (9.279,4.556) circle (3pt);

        \filldraw[blue] (2.82,1.396) circle (3pt);
        \filldraw[blue] (4.775,3.908) circle (3pt);
        \filldraw[blue] (7.325,4.56) circle (3pt);
        \filldraw[blue] (9.279,6.328) circle (3pt);

        \filldraw[yellow] (2.82,1.5) circle (3pt);
        \filldraw[yellow] (4.775,4.2) circle (3pt);
        \filldraw[yellow] (7.325,4.9) circle (3pt);
        \filldraw[yellow] (9.279,6.8) circle (3pt);
        \filldraw[blue] (9.953,0.555) circle (3pt);
        \filldraw[blue] (10.526,0.555) circle (3pt);
        \filldraw[blue] (11.274,0.555) circle (3pt);
        \filldraw[blue] (11.847,0.555) circle (3pt);

        \filldraw[blue] (9.953,2.64) circle (3pt);
        \filldraw[blue] (10.526,2.64) circle (3pt);
        \filldraw[blue] (11.274,2.64) circle (3pt);
        \filldraw[blue] (11.847,2.64) circle (3pt);

        \filldraw[blue] (9.953,5.36) circle (3pt);
        \filldraw[blue] (10.526,5.36) circle (3pt);
        \filldraw[blue] (11.274,5.36) circle (3pt);
        \filldraw[blue] (11.847,5.36) circle (3pt);

        \filldraw[blue] (9.953,7.444) circle (3pt);
        \filldraw[blue] (10.526,7.444) circle (3pt);
        \filldraw[blue] (11.274,7.444) circle (3pt);
        \filldraw[blue] (11.847,7.444) circle (3pt);

    \end{tikzpicture}
    \caption{Quadrature on interface element $K = [\tilde{x}_m,\tilde{x}_M] \times [\tilde{y}_m,\tilde{y}_M]$. $r_1,r_2$ are roots for functions $\tilde{f}_m, \tilde{f}_M$ respectively. Volume quadratures: blue nodes $(\bar{\bm{r}} \; \vert \; \check{\bm{\rho}})$ for $\Omega_-$ and red nodes $(\bar{\bm{r}} \; \vert \; \check{\bm{\rho}})$ for $\Omega_+$. Surface quadrature on $\Gamma$: yellow nodes, $\bm{\rho}$, also $f_j^*$'s root for each $j$.}
    \label{fig:InterfElemQuad}
\end{figure}
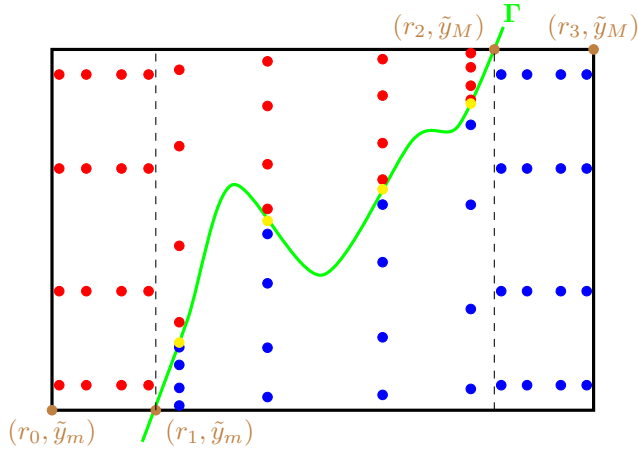

To begin, we use standard techniques for regular elements $K \in \Omega_{\pm,h} \backslash \mathcal{T}_{\Gamma,h}$, 2D Gauss-Legendre quadrature being our method of choice.

However, when dealing specifically with interface elements, $K \in \mathcal{T}_{\Gamma,h}$, a more sophisticated algorithm is required. This is because our integral regions are arbitrary and implicitly defined through a level-set. Two main approaches exist in this context: one can either adjust the weights or reposition the nodes to best fit the irregular integration domain.

In our implementation, we choose an algorithm developed by Saye, as described in \cite{Saye2015}, which falls into the latter category. This choice was motivated by the fact that we are dealing with especially high-order polynomials. Indeed, for $p$-convergence, we quickly encounter machine precision issues or limitations, and methods that aim to adjust weights require a significantly larger number of nodes to achieve comparable accuracy. Furthermore, other available techniques were originally designed with simplex elements in mind, making their adaptation and implementation for our rectangular meshes less nature.

We only summarize the method here, specifically for 2D, and refer interested readers to \cite{Saye2015} for full details and generality. A crucial requirement for the algorithm is for $\Gamma$ to be a local graph, i.e., $\Gamma \vert_K = \Phi(x,y), \; (x,y) \in K$ for any $K \in \mathcal{T}_{\Gamma,h}$. Hence, our numerical grid must be fine enough with respect to the problem's geometry. From Assumption \ref{ass:InterfIntersects}, only a minor extra refinement may be needed, in the worst case scenario, to satisfy the graph condition. For context, we re-state that our elements are two-dimensional, square-shaped, and the basis functions defined through interpolation about the 2D LGL points.

\begin{remark}
    We strictly use Gauss-Legendre nodes for quadratures both for interface and non-interface elements. Higher precision and no collapsing issues make them superior. Indeed, LGL points collapse on top of each other, at $(r_1,\tilde{y}_m)$ for example, if used in Algorithm \ref{alg:capTest}.
\end{remark}
\begin{remark}
    We use Ridder's method \cite{Ridders1979}, based on Regula Falsi and the exponential function, for root-finding in pseudo-code Algorithm  \ref{alg:capTest}. 
\end{remark}

\begin{algorithm}[H]
\caption{Interface Element Quadratures}\label{alg:capTest}
\begin{algorithmic}[1]
\STATE Let $\Gamma = \psi(x,y)$ and  $(z_j,w_j)_{j=0}^L$ be the 1D Gauss-Legendre quadrature on reference interval $[-1,1]$ of order $2L+1$. 
\FOR{$T \in \mathcal T_{\Gamma,h}$}
    \STATE $T = [x_m,x_M] \times [y_m,y_M]$.
    \STATE $x_c = \frac{x_m+x_M}{2}, \; y_c = \frac{y_m+y_M}{2}$. \COMMENT{Choose independent and dependent coordinates to avoid null derivative and be able to use Implicit Function Theorem.}
    \IF{$|\psi_x(x_c,y_c)| \geq |\psi_y(x_c,y_c)|$}
        \STATE $(\tilde{x},\tilde{y}) = (x,y)$
    \ELSE
        \STATE $(\tilde{x},\tilde{y}) = (y,x)$
    \ENDIF
    \STATE Define two new functions $\tilde{f}_m(\tilde{x}) \coloneqq \psi(\tilde{x},\tilde{y}_m)$ and $\tilde{f}_M(\tilde{x}) \coloneqq \psi(\tilde{x},\tilde{y}_M)$ and compute their root, $r$, on $[\tilde{x}_m,\tilde{x}_M]$. 
    \STATE We now have a partition $[\tilde{x}_m,\tilde{x}_M] = \cup_{i=0}^{N-1} [r_i,r_{i+1}]$ where $r_0 = \tilde{x}_m$ and $r_N = \tilde{x}_M$.
    \FOR{$i = 0 \hdots N-1$}
        \IF{$[r_i,r_{i+1}] \times [\tilde{y}_m,\tilde{y}_M] \cap \Gamma = \emptyset$}
            \STATE Apply 2D Gauss-Legendre quadrature on $[r_i,r_{i+1}] \times [\tilde{y}_m,\tilde{y}_M]$.
        \ELSE 
            \STATE Apply 1D Gauss-Legendre quadrature on $[r_i,r_{i+1}]$ i.e. 
            $$\bm{r}^* = \{r^*_j\}_{j=0}^L = \left\{ r_i + \frac{r_{i+1}-r_i}{2} (z_j+1) \right\}_{j=0}^L. $$ \label{if1dGauss}
            \FOR{$j = 0 \hdots L$}
                \STATE Define $f_j^*(\tilde{y}) \coloneqq \psi(r_j^*,\tilde{y})$ and compute its root, $\rho_j$, on the interval $[\tilde{y}_m,\tilde{y}_M]$.
                \STATE Apply 1D Gauss-Legendre quadrature on $[\tilde{y}_m,\rho_j]$ and also on $[\rho_j,\tilde{y}_M]$ i.e.
                $$ \check{\bm{\rho}}_j = \left\{ \tilde{y}_m + \frac{\rho_j-\tilde{y}_m}{2}(z_k+1) \right\}_{k=0}^L, \, \hat{\bm{\rho}}_j = \left\{ \rho_j + \frac{\tilde{y}_M-\rho_j}{2}(z_k+1) \right\}_{k=0}^L. $$
            \ENDFOR
            \STATE For volume integrals, $(\bar{\bm{r}} \; \vert \; \check{\bm{\rho}})$, $(\bar{\bm{r}} \; \vert \; \hat{\bm{\rho}})$, both in $\mathbb R^{(L+1)^2 \times 2}$, are quadrature nodes for regions $T \cap \Omega_-$, $T \cap \Omega_+$ respectively. 
            \begin{gather*}
                \bar{\bm{r}} \coloneqq (\bm{r}^*,\hdots,\bm{r}^*), \; \check{\bm{\rho}} = (\check{\bm{\rho}}_0,\hdots,\check{\bm{\rho}}_L), \; \hat{\bm{\rho}} = (\hat{\bm{\rho}}_0,\hdots,\hat{\bm{\rho}}_L), \text{ all in } \mathbb R^{(L+1)^2}.
            \end{gather*}
            Weights are scaled accordingly with $\frac{r_{i+1}-r_i}{2}$, and potential values $\frac{\rho_j-\tilde{y}_m}{2}$ or $\frac{\tilde{y}_M-\rho_j}{2}$. 
            \STATE For surface integrals, nodes are $(\bm{r}^* \; \vert \; \bm{\rho}) \in \mathbb R^{(L+1) \times 2}$, $\bm{\rho} = \{\rho_j\}_{j=0}^L$ and
            \STATE weights, after accounting for change of measure from $dS$ to $d\tilde{x}$, are 
            $$\bm{\omega} = \left\{ w_j \frac{|\nabla\psi(r_j^*,\rho_j)|}{|\partial_{\tilde{y}} \psi(r_j^*,\rho_j)|} \right\}_{j=0}^L.$$ \label{alg:wtsRescl} 
        \ENDIF
    \ENDFOR
\ENDFOR
\end{algorithmic}
\end{algorithm}

Surface quadrature weights rescaling comes from the local graph condition and implicit function theorem. Indeed, if $\tilde{y} = \zeta_K(\tilde{x})$, on $\Gamma \cap K$ for any $K \in \mathcal T_{\Gamma,h}$, then we have
\begin{gather*}
    \frac{|\nabla \psi|}{|\partial_{\tilde{y}}\psi|} = \sqrt{1 + \left( \frac{\partial_{\tilde{x}}\psi}{\partial_{\tilde{y}}\psi} \right)^2} = \sqrt{1 + (\zeta_K'(\tilde{x}))^2}, 
    \end{gather*}
    and hence
    \begin{align*}
       \int_{K\cap\Gamma} F \; dS &= \int_{r_i}^{r_{i+1}} F(\tilde{x},\zeta_K(\tilde{x})) \sqrt{1 + (\zeta_K'(\tilde{x}))^2} \; d\tilde{x}, \\ 
        &= \int_{r_i}^{r_{i+1}} F(\tilde{x},\zeta_K(\tilde{x})) \frac{|\nabla \psi|}{|\partial_{\tilde{y}}\psi|} \; d\tilde{x}.
    \end{align*}

The question now becomes how many quadrature nodes are necessary to avoid, or at least minimize, introducing additional error into our numerical algorithm. Firstly, we consider integrands consisting solely of our polynomial basis functions. For volume integrals, the highest degree polynomial comes from the term $(u,v)_\Omega$, present in (\ref{equ:extend_mass}). We have
\begin{gather*}
    u,v \in V^{p,h} \text{ hence } u,v = O(x^py^p) \implies uv = O(x^{2p}y^{2p}).
\end{gather*}
Therefore the highest degree we encounter has degree $4p$. Now from numerical quadrature theory, we know monomials $x^i y^j$ are integrated exactly with $(L+1)^2$ Gauss-Legendre nodes if $0 \leq i+j < 2(L+1)$. It follows that we need $L \geq 2p$ as $2(L+1) \geq 2(2p+1) > 4p$. In the surface context, we consider the term $(\llbracket u \rrbracket,\llbracket v \llbracket)_\Gamma$, part of the bilinear form in (\ref{equ:bilinear}). The weight rescaling in step \ref{alg:wtsRescl} in Algorithm \ref{alg:capTest} does not contribute any error as we explicitly know the level-set $\Gamma = \psi(x,y)$. Again,
\begin{gather*}
    u,v \in V^{h,p} \text{ so } \llbracket u \rrbracket,\llbracket v \rrbracket = O(x^py^p) \implies \llbracket u \rrbracket \llbracket v \rrbracket = O(x^{2p}y^{2p}).
\end{gather*}
Now through the change of variables detailed in Step \ref{alg:wtsRescl}, we know the integration will depend exclusively on the independent coordinate, $\tilde{x}$. 
We end up with 
\begin{gather*}
    \llbracket u \rrbracket \llbracket v \rrbracket = O(\tilde{x}^{2p} \zeta_K(\tilde{x})^{2p}) \text{ on } \Gamma.
\end{gather*} 
Here we cannot guarantee an exact result as the map $\zeta_K$ is arbitrary. Under the assumption that $\zeta_K$ is simpler than, or does not deviate too far from, a linear function, then
\begin{equation*}
    O(\tilde{x}^{2p} \zeta_K(\tilde{x})^{2p}) \approx O(\tilde{x}^{4p}).
\end{equation*}
Gaussian Quadrature shall be very precise as monomials are exactly integrated with $L+1 = 2p+1$ quadrature nodes since $4p < 2(L+1) = 4p+2$. Additionally, face integrals present in the ghost penalty (\ref{equ:ghostpenalty}) are exactly handled because the normal derivative terms are lower order polynomials than $4p$. We conclude that at the very least we require $O((2p+1)^2)$  and $(2p+1)$ nodes for volume and surface integrals respectively. Lastly, for more general integrands, where either the source term or boundary/interface conditions appear, we carried out numerical tests with $O(N^2)$ and $N$ number of volume/surface quadrature nodes for different $N$ values such as $N = 3p > 2p+1$. We observed no significant accuracy improvements and therefore kept the original $(2p+1)^2$ and $(2p+1)$ schemes.

\section{Error estimate}
\label{sec:error}
In this section, we shall carry out our $hp$ error analysis by establishing  convergence rates with respect to both $h$ and $p$ parameters. Before doing so, we quantify how consistency deviates due to ghost penalty term. In particular, we can show the following weak Galerkin orthogonality.
\begin{theorem}\label{thm:weakgal}
	Let $u \in H^2(\Omega_+\cup\Omega_-) \cap H_0^1(\Omega)$ be the solution of the interface problem \eqref{equ:bvp}  and $u^{p,h} \in V^{p,h}$ be the solution of \eqref{equ:usem_bvp}. Then, it holds that
\begin{equation}\label{equ:weakgalerkin}
    a_{p,h}(u-u^{p,h}, v) = \frac{\gamma_A}{h^2}g_{p,h}(u^{p,h}, v), \quad \forall v \in V^{p,h}.
\end{equation}
\end{theorem}
\begin{proof}
 It follows from the fact that $a_{p,h}(u, v) = l_{p,h}(v)$ for any $v \in V^{p,h}$. 
\end{proof}

\subsection{Error estimates for interface problems}
To prepare the a priori error estimate for interface problems, we shall introduce the extension operator $E_{i}$ for $i=\pm$. For any function $v_i\in H^k(\Omega_i)$, its extension $E_{i}v_i$ is  a function in $H^{k}(\Omega)$ satisfying $(E_{i}v_i)|_{\Omega_i} = v_i$ and $\|E_{i}v_i\|_{k, \Omega}\lesssim \|v_i\|_{k, \Omega_i}$. In the proofs below, we shall often abuse notation and interchangeably use, 
\begin{equation*}
    v = (v_-,v_+) \in H^k(\Omega_-) \times H^k(\Omega_+) \text{ and } Ev \coloneqq (E_-v_-,E_+v_+) \in H^k(\Omega) \times H^k(\Omega).
\end{equation*}

Let $I^{p,h}_i: H^1(\Omega_{i,h})\rightarrow V^{p,h}(\Omega_{i,h}) $ denote the Lagrange-Gauss-Lobatto polynomial interpolation operator, as  defined in \cite{BabuskaSuri1987}. We extend this operator to the unfitted spectral element space $V^{p,h}$, denoted as $I^{p,h}$, with the following expression:
\begin{equation}\label{equ:interpolation}
    I^{p,h}v = (I^{p,h}_- E_-v, I^{p,h}_+E_+v) \in V^{p,h}.
\end{equation}
As established by \cite[Theorem 4.5]{BabuskaSuri1987}, the subsequent approximation property is valid:
\begin{equation}\label{equ:interpapp}
    \|v - I^{p,h}v\|_{j, K} \lesssim  \frac{h^{\min(p+1,k)-j}}{p^{k-j}} \|v\|_{k, K}, \quad \forall v \in H^{k}(K),
\end{equation}
where $K\in \mathcal{T}_h$ and $0\le j \le k$.

Furthermore, for any function $v \in H^1\left(\Omega_{-,h} \cup \Omega_{+,h} \right)$, we recall the trace inequalities presented in \cite{GMassing2019}:
\begin{align}
    &\|v\|_{0, \partial K} \lesssim h^{-1/2}\|v\|_{0, K} + h^{1/2}\|\nabla v\|_{0, K}, \quad \forall K \in \mathcal{T}_h, \label{equ:trace}\\
    &\|v\|_{0, \Gamma \cap \partial K} \lesssim h^{-1/2}\|v\|_{0, K} + h^{1/2}\|\nabla v\|_{0, K}, \quad \forall K \in \mathcal{T}_{\Gamma, h}, \label{equ:intertrace}.
\end{align}

We initiate our error analysis by addressing the consistency error arising from the ghost penalty term.
\begin{theorem}\label{thm:ghostapp}
    Let $I^{p,h}$ be the LGL polynomial interpolator operator defined in \eqref{equ:interpolation}. Suppose $v\in H^{k}(\Omega_+\cup \Omega_-)$, for $k \geq 2$. Then, the following estimate holds:
    \begin{equation}\label{equ:ghostapp}
        g_{p,h}(I^{p,h}v, I^{p,h}v) \lesssim \frac{h^{2\min(p+1,k)}}{p^{2(k-1)}} \|v\|_{k, \Omega_+\cup \Omega_-}^2.
    \end{equation}
\end{theorem}
\begin{proof}
    Let $m = \min(k,p+1)$. We can deduce that:
    \begin{equation} \label{equ:ghostapp2}
    \begin{aligned}
                &g_{p,h}(I^{p,h}v, I^{p,h}v) \\
                =&\sum_{e \in \mathcal{G}_{\pm,h}} \sum_{j=1}^{p}\frac{h^{2j+1}}{p^{2j}} ([ \partial^j_{\mathbf{n}} (I^{p,h}v)], [\partial^j_{\mathbf{n}} (I^{p,h}v)])_e \\
                =&\sum_{e \in \mathcal{G}_{\pm,h}} \sum_{j=1}^{m-1}\frac{h^{2j+1}}{p^{2j}} ([ \partial^j_{\mathbf{n}} (v-I^{p,h}v) ], [ \partial^j_{\mathbf{n}} (v-I^{p,h}v) ])_e \\
                &+ \sum_{e \in \mathcal{G}_{\pm,h}} \sum_{j=m}^{p} \frac{h^{2j+1}}{p^{2j}} ([ \partial^j_{\mathbf{n}} (I^{p,h}v) ], [ \partial^j_{\mathbf{n}} (I^{p,h}v) ])_e \\
                \le& \sum_{e \in \mathcal{G}_{\pm,h}} \left( \sum_{j=1}^{m-1}\frac{h^{2j+1}}{p^{2j}} \|[ \partial^j_{\mathbf{n}} (v-I^{p,h}v) ]\|_{0, e}^2 + \sum_{j=m}^{p} \frac{h^{2j+1}}{p^{2j}} \|[ \partial^j_{\mathbf{n}} (I^{p,h}v) ]\|_{0, e}^2 \right) \\
                \lesssim &\sum_{K \in \Omega_{\Gamma, h}} \left( \sum_{j=1}^{m-1} \frac{h^{2j+1}}{p^{2j}} \left( \frac{1}{h} \| D^{j} (v-I^{p,h}v)\|_{0, K}^2 + 
                  h\| D^{j+1} (v-I^{p,h}v)\|_{0, K}^2 \right) \right. \\
                &+ \left. \sum_{j=m}^{p} \frac{h^{2j+1}}{p^{2j}} \left( \frac{1}{h} \| D^{j} I^{p,h}v\|_{0, K}^2 + h \| D^{j+1} I^{p,h}v \|_{0, K}^2 \right) \right) \\
                \lesssim &\sum_{K \in \Omega_{\Gamma, h}} \left( \sum_{j=1}^{m-1} \frac{h^{2j+1}}{p^{2j}} \left( \frac{1}{h} \frac{h^{2\min(p+1,k)-2j}}{p^{2k-2j}} + h \frac{h^{2\min(p+1,k)-2j-2}}{p^{2k-2j-2}} \right) \|v\|_{k, K}^2 \right. \\
                &+ \left. \sum_{j=m}^{p} \frac{h^{2j+1}}{p^{2j}} \left( \frac{1}{h} \frac{h^{2\min(p+1,k)-2j}}{p^{2k-2j}} + 
                  h\frac{h^{2\min(p+1,k)-2j-2}}{p^{2k-2j-2}} \right) \|I^{p,h}v\|_{k, K}^2 \right) \\
                \lesssim &\frac{h^{2\min(p+1,k)}}{p^{2(k-1)}} \|v\|_{k, \Omega_+\cup \Omega_-}^2, 
    \end{aligned}
    \end{equation}
    where we have used the trace inequality \eqref{equ:trace} in the second inequality and the interpolation approximation estimate \eqref{equ:interpapp} in the third inequality. This completes the proof.
\end{proof}

With the above consistency error, we can now proceed to establish the approximation error in the energy norm. 
\begin{theorem}\label{thm:interperr}
    Under the same assumption as in Theorem \ref{thm:ghostapp}, we have the following estimate:
    \begin{equation}\label{equ:energyapp}
        \vertiii{v-I^{p,h}v} \lesssim \frac{h^{\min(p+1,k)-1}}{p^{k-\frac{3}{2}}} \|v\|_{k, \Omega_+\cup \Omega_-}.
    \end{equation}
\end{theorem}

\begin{proof}
    We recall the energy norm in \eqref{equ:augmentedenergynorm},
    \begin{equation}
    \begin{aligned}
            \vertiii{v-I^{p,h}v} 
            \lesssim & \|\nabla (v-I^{p,h}v) \|_{0, \Omega_- \cup \Omega_+} + \| \dgal{\partial_{\mathbf{n}} (v-I^{p,h}v)} \|_{-1,p,h,\Gamma} \\
            +& \| \llbracket v-I^{p,h}v \rrbracket \|_{1,p, h,\Gamma} \\
            \coloneqq & I_1 + I_2 + I_3.
    \end{aligned}
    \end{equation}
    
    According to \eqref{equ:interpapp} and Theorem \ref{thm:ghostapp}, one has
    \begin{equation}
    I_1 \lesssim \frac{h^{\min(p+1,k)-1}}{p^{k-1}} \|v\|_{k, \Omega_+\cup\Omega_-}.
    \end{equation}
    For $I_2$, applying the trace inequality \eqref{equ:intertrace} gives
    \begin{equation}
    \begin{aligned}
        I_2^2 = &\frac{h}{p^2} \sum_{K \in \mathcal{T}_{\Gamma,h}}\|\dgal{\partial_{\mathbf{n}}(v-I^{p,h}v)}\|_{0, \Gamma_K}^2 \\
        \lesssim & \frac{h}{p^2} \sum_{K\in\mathcal{T}_{\Gamma,h}} \left( \frac{1}{h}\|v - I^{p,h}v\|_{1, K}^2	+ h \|v - I^{p,h}v\|_{2, K}^2 \right) \\
        \lesssim & \frac{h}{p^2} \sum_{K \in \mathcal{T}_{\Gamma,h}} \left( \frac{1}{h}\frac{h^{2\min(p+1,k)-2}}{p^{2k-2}} \|v\|_{k, K}^2 	+ h \frac{h^{2\min(p+1,k)-4}}{p^{2k-4}} \|v\|_{k, K}^2 \right) \\
        \lesssim & \sum_{K \in \mathcal{T}_{\Gamma,h}} \left( \frac{h^{2\min(p+1,k)-2}}{p^{2k}} \|v\|_{k, K}^2 + \frac{h^{2\min(p+1,k)-2}}{p^{2k-2}} \|v\|_{k, K}^2 \right) \\
        \lesssim & \frac{h^{2\min(p+1,k)-2}}{p^{2(k-1)}} \|v\|_{k, \Omega_+\cup\Omega_-}^2,
    \end{aligned}
    \end{equation}
    where we have used the interpolation approximation property \eqref{equ:interpapp} in the first inequality and the property of the extension operator in the last inequality. Similarly, we can show that:
    \begin{equation}
    \begin{aligned}
        I_3 \lesssim &\frac{h^{\min(p+1,k)-1}}{p^{k-\frac{3}{2}}} \|v\|_{k, \Omega_+\cup\Omega_-}.
    \end{aligned}
    \end{equation}
    Combining the above error estimates, we conclude the proof. 
\end{proof}

Then, we present an intermediate result on the USEM solution and the interpolation of the exact solution.\begin{theorem}\label{thm:interpenergyest}
    Let $u$ be the solution of the interface problem \eqref{equ:bvp}, and $u^{p,h}$ be its unfitted spectral element solution. Suppose $u\in H^{k}(\Omega_+\cup \Omega_-)$, for $k \geq 2$. Then we have the following estimate:
    \begin{equation}\label{equ:energyest}
        \vertiii{u^{p,h}-I^{p,h}u}_{\ast} \lesssim \frac{h^{\min(p+1,k)-1}}{p^{k-\frac{3}{2}}} \|u\|_{k, \Omega_+\cup \Omega_-}.
    \end{equation}
\end{theorem}
\begin{proof}
    We decompose the error $e = u^{p,h} - u$ into two components, i.e.,
    \begin{equation}\label{equ:decomp}
        e_h \coloneqq u^{p,h} - I^{p,h}u, \quad e_I \coloneqq I^{p,h} u - u.
    \end{equation}
    Using the coercivity \eqref{equ:coercivity}, the weak Galerkin orthogonality \eqref{equ:weakgalerkin}, and the Cauchy-Schwartz inequality produces 
    \begin{equation}
    \begin{aligned}
        \vertiii{e_h}_{\ast}^2 \lesssim & a_{p,h}(e_h, e_h) + \frac{\gamma_A}{h^2}g_{p,h}(e_h, e_h) \\
        = & a_{p,h}(u^{p,h}-u, e_h) - a_{p,h}(e_I, e_h) + \frac{\gamma_A}{h^2}g_{p,h}(e_h, e_h) \\
        = & -\frac{\gamma_A}{h^2}g_{p,h}(u^{p,h}, e_h) - a_{p,h}(e_I, e_h) + \frac{\gamma_A}{h^2}g_{p,h}(e_h, e_h) \\
        = & -a_{p,h}(e_I, e_h) - \frac{\gamma_A}{h^2}g_{p,h}(I^{p,h}u, e_h) \\
        \lesssim & \left(\vertiii{e_I}\vertiii{e_h} + \frac{1}{h}g_{p,h}(I^{p,h}u, I^{p,h}u)^{\frac{1}{2}}\frac{1}{h}g_{p,h}(e_h, e_h)^{\frac{1}{2}}\right) \\
        \lesssim & \left(\vertiii{e_I} + \frac{1}{h}g_{p,h}(I^{p,h}u, I^{p,h}u)^{\frac{1}{2}}\right)\vertiii{e_h}_{\ast}.
    \end{aligned}
    \end{equation}
    This implies that:
    \begin{equation} \label{equ:intererr}
        \vertiii{e_h}_{\ast} \lesssim \vertiii{e_I} + \frac{1}{h}g_{p,h}(I^{p,h}u, I^{p,h}u)^{\frac{1}{2}} \lesssim \frac{h^{\min(p+1,k)-1}}{p^{k-\frac{3}{2}}} \|v\|_{k, \Omega_+\cup \Omega_-}.
    \end{equation}
    The result follows from Theorems \ref{thm:ghostapp} \& \ref{thm:interperr}. 
\end{proof}

Now, we are ready to present our results in energy norm. 
\begin{theorem}\label{thm:energyerr}
    Under the same assumptions as in Theorem \ref{thm:interpenergyest}, we have the following estimate:
    \begin{equation}\label{equ:energyest2}
        \vertiii{u^{p,h}-u} \lesssim \frac{h^{\min(p+1,k)-1}}{p^{k-\frac{3}{2}}} \|u\|_{k, \Omega_+\cup \Omega_-}.
    \end{equation}
\end{theorem}
\begin{proof}
    \begin{equation}
    \begin{aligned}
        \vertiii{u^{p,h}-u} \leq& \vertiii{u^{p,h}-I^{p,h}u} + \vertiii{I^{p,h}u-u} \\
        \leq& \vertiii{u^{p,h}-I^{p,h}u}_\ast + \vertiii{I^{p,h}u-u} \\
        \lesssim & \frac{h^{\min(p+1,k)-1}}{p^{k-\frac{3}{2}}} \|v\|_{k, \Omega_+\cup \Omega_-}.
    \end{aligned}
    \end{equation}
    The result follows from Theorem \ref{thm:interperr} and Theorem \ref{thm:interpenergyest}. 
\end{proof}

We end this subsection by establishing the error estimate in the $L^2$ norm using the Aubin-Nitsche argument. For this purpose, we introduce the dual interface problem:
\begin{gather}\label{equ:dual}
\begin{cases}
-\nabla \cdot (\alpha \nabla \phi) = (u-u^{p,h}) \text{ in } \Omega, \\
\phi= 0 \text{ on } \partial \Omega, \\
\llbracket \phi \rrbracket = \llbracket \alpha \partial_{\mathbf{n}} \phi \rrbracket = 0 \text{ on } \Gamma.
\end{cases}
\end{gather}
The regularity result implies that:
\begin{equation}\label{equ:regularity}
	\|\phi\|_{2, \Omega_+\cup\Omega_-} \lesssim \|u-u^{p,h}\|_{0, \Omega}. 
\end{equation}

\begin{theorem}\label{thm:l2error}
	Under the same assumptions as in Theorem \ref{thm:interpenergyest}, there holds
	\begin{equation}\label{equ:l2est}
		\|u-u^{p,h}\|_{0,\Omega} \lesssim \frac{h^{\min(p+1,k)}}{p^{k-1}} \|u\|_{k, \Omega_+\cup \Omega_-}.
	\end{equation}
\end{theorem}
\begin{proof}
	Multiplying the first equation in \eqref{equ:dual} with $u-u^{p,h}$ and applying Green's formula, we have:
	\begin{equation}
		\begin{aligned}
			&\|u-u^{p,h}\|_{0,\Omega}^2 \\
			= & a_{p,h}(u- u^{p,h}, \phi) \\
			= & a_{p,h}(u-u^{p,h}, \phi-I^{p,h}\phi) + a_{p,h}(u-u^{p,h}, I^{p,h}\phi) \\
			= & a_{p,h}(u-u^{p,h}, \phi-I^{p,h}\phi) + \frac{\gamma_A}{h^2}g_{p,h}(u^{p,h}, I^{p,h}\phi)\\
			\lesssim & \vertiii{u-u^{p,h}}    \vertiii{\phi-I^{p,h}\phi} +
			\frac{1}{h} g_{p,h}(u^{p,h}, u^{p,h})^{\frac{1}{2}} 
			\frac{1}{h} g_{p,h}( I^{p,h}\phi,  I^{p,h}\phi)^{\frac{1}{2}}\\
			\coloneqq & I_1 + I_2. 
		\end{aligned}
	\end{equation}
	For $I_1$, using Theorem \ref{thm:interperr}, Theorem \ref{thm:energyerr},  and \eqref{equ:regularity}, we obtain:
	\begin{equation}
		I_1 \le \frac{h^{\min(p+1,k)}}{p^{k-1}}\|u-u^{p,h}\|_{0, \Omega} \|u\|_{k, \Omega_+\cup\Omega_-}.
	\end{equation}
	To estimate $I_2$, we use Theorem \ref{thm:ghostapp} and observe:
	\begin{equation}\label{equ:other}
	\begin{aligned}
        \frac{1}{h^2} g_{p,h}(u^{p,h}, u^{p,h}) \lesssim & \vertiii{u^{p,h} - I^{p,h}u}_\ast^2 + \frac{1}{h^2} g_{p,h}(I^{p,h}u,I^{p,h}u) \\ 
        + & 2 \vertiii{u^{p,h} - I^{p,h}u}_\ast \frac{1}{h} g_{p,h}(I^{p,h}u,I^{p,h}u)^{\frac{1}{2}} \\
		\lesssim & \left( \vertiii{u^{p,h}-I^{p,h}u}_{\ast}
		+ \frac{1}{h} g_{p,h}(I^{p,h}u, I^{p,h}u)^\frac{1}{2} \right)^2 \\
		\lesssim & \frac{h^{2\min(p+1,k)-2}}{p^{2k-3}} \|u\|_{k, \Omega_+\cup \Omega_-}^2.
	\end{aligned}
	\end{equation}
	Theorem \ref{thm:ghostapp} and regularity \eqref{equ:regularity} imply:
	\begin{equation}
		\frac{1}{h} g_{p,h}(I^{p,h}\phi,  I^{p,h}\phi)^{\frac{1}{2}} \lesssim \frac{h}{p^{1/2}} \|u-u^{p,h}\|_{0, \Omega}.
	\end{equation}
	Combining the above estimates completes the proof of \eqref{equ:l2est}.
\end{proof}

\subsection{Error estimates for interface eigenvalue problems} In this section, we adopt the Babu{\v{s}}ka-Osborne theory [1, 2] to establish the approximation results for interface eigenvalue problems. For this purpose, we define the solution operator $T: L^2(\Omega) \rightarrow  H^2(\Omega_+\cup\Omega_-) \cap H_0^1(\Omega)$ as
\begin{equation}\label{equ:soloperator}
    a_{p,h}(Tf, v) = (f, v), \quad  \forall v \in H^\frac{3}{2}(\Omega_+\cup\Omega_-).
\end{equation}
The eigenvalue problem \eqref{equ:eigenprob} can be rewritten as
\begin{equation}
    Tu = \mu u.
\end{equation}
Analogously, we can define the discrete solution operator $T^{p,h}: V^{p,h} \rightarrow V^{p,h}_0$ as
\begin{equation}\label{equ:dissoloperator}
    A_{p,h}(T^{p,h}f, v) = M_{p,h}(f, v), \quad \forall v \in V^{p,h}.
\end{equation}
The discrete eigenvalue problem \eqref{equ:usem_eigen} is equivalent to
\begin{equation}
    T^{p,h} u^{p,h} = \mu^{p,h} u^{p,h}.
\end{equation}
It is straightforward to see that we have $\mu = \frac{1}{\lambda}$ and $\mu^{p,h} = \frac{1}{\lambda^{p,h}}$. Furthermore, $T^{p,h}$ is self-adjoint.

To facilitate our analysis, we introduce an intermediate interface eigenvalue problem: find $(\tilde{\lambda}^{p,h},\tilde{u}^{p,h}) \in \mathbb R^+ \times V^{p,h}_0$ such that
\begin{equation}\label{equ:iusem_eigen}
    A_{p,h}(\tilde{u}^{p,h}, v) = \tilde{\lambda}^{p,h}(\tilde{u}^{p,h}, v), \quad \forall v \in V^{p,h}.
\end{equation}
The corresponding solution operator $\tilde{T}^{p,h}: L^2(\Omega) \rightarrow V^{p,h}_0$ is defined as
\begin{equation}\label{equ:isoloperator}
    A_{p,h}(\tilde{T}^{p,h}f, v) = (f, v), \quad \forall v \in V^{p,h}_0.
\end{equation}
Also, $\tilde{T}^{p,h}$ is self-adjoint.

We commence our analysis with the following theorem concerning approximation:
\begin{theorem}\label{thm:opapp}
	Let $\mu, \; E_{\mu}$ denote $T$'s eigenvalue  and eigenspace respectively. Suppose $\mu$ has multiplicity $m$ and $E_{\mu} \subset H^{k}(\Omega_+\cup \Omega_-)$ with $k\ge p+1$. Then, we have 
	\begin{equation}\label{equ:operatornorm}
	\|(T- T^{p,h})|_{E_{\mu}}\|_{\mathcal{L}(L^2(\Omega))} \lesssim \frac{h^{\min(p+1,k)-1}}{p^{k-1}}.
	\end{equation}
\end{theorem}
\begin{proof} 
	The assumption $E_{\mu} \subset H^{k+1}(\Omega_+\cup \Omega_-)$ suggests that $T^{p,h}$ is well-defined. For any $u \in E_{\mu}$ with $\|u\|_{0, \Omega}=1$, the triangle inequality implies that 
	\begin{equation}
		\|(T - T^{p,h})u\|_{0, \Omega} \le \|(T - \tilde{T}^{p,h})u \|_{0, \Omega} + \|( \tilde{T}^{p,h} - T^{p,h}) u \|_{0, \Omega} \coloneqq I_1 + I_2. 
	\end{equation}
	For $I_1$, we can bound it using Theorem \ref{thm:l2error} as follows:
	\begin{equation}
		I_1 \le \frac{h^{\min(p+1,k)}}{p^{k-1}} \|u\|_{k, \Omega_+\cup \Omega_-}. 
	\end{equation}
	
	To bound $I_2$, we first consider its difference in energy norm:
	\begin{equation}
	\begin{aligned}
	& \vertiii{(\tilde{T}^{p,h} - T^{p,h})u}_{\ast}^2 \\
		  \lesssim & A_{p,h}((\tilde{T}^{p,h} - T^{p,h})u, (\tilde{T}^{p,h} - T^{p,h})u)\\
		   = & A_{p,h}(\tilde{T}^{p,h}u, (\tilde{T}^{p,h} - T^{p,h})u) - A_{p,h}(T^{p,h}u, (\tilde{T}^{p,h} - T^{p,h})u)\\
		   =&  g_{p,h}(u-I^{p,h}u, (\tilde{T}^{p,h} - T^{p,h})u) +  g_{p,h}(I^{p,h}u, (\tilde{T}^{p,h} - T^{p,h})u)\\
		   \lesssim & \left[ g_{p,h}(u-I^{p,h}u, u-I^{p,h}u)^{1/2}+ g_{p,h}(I^{p,h}u, I^{p,h}u)^{1/2} \right] \\
		   & g_{p,h}((\tilde{T}^{p,h} - T^{p,h})u, (\tilde{T}^{p,h} - T^{p,h})u)^{1/2} \\
		   \lesssim &  g_{p,h}(I^{p,h}u, I^{p,h}u)^{1/2}\vertiii{(\tilde{T}^{p,h} - T^{p,h})u}_{\ast}. 
	\end{aligned}
	\end{equation}
	which implies that 
	\begin{equation}\label{equ:energydiff}
		\vertiii{(\tilde{T}^{p,h} - T^{p,h})u}_{\ast} \le  g_{p,h}(I^{p,h}u, I^{p,h}u)^{1/2}\lesssim
		\frac{h^{\min(p+1,k)}}{p^{k-1}} \|u\|_{k, \Omega_+\cup \Omega_-},
	\end{equation}
	where we have used the approximation result for the ghost penalty term in Theorem \ref{thm:ghostapp}.  Then, Poincaré's inequality shows that 
	\begin{equation}
		\|\tilde{T}^{p,h}u - T^{p,h}u \|_{0, \Omega}\lesssim
		\frac{h^{\min(p+1,k)}}{p^{k-1}} \|u\|_{k, \Omega_+\cup \Omega_-}.
	\end{equation}
	
	Combining the above estimates gives
	\begin{equation}
		\|Tu - T^{p,h}u \|_{0, \Omega}\lesssim
		\frac{h^{\min(p+1,k)-1}}{p^{k-1}} \|u\|_{k, \Omega_+\cup \Omega_-}. 
	\end{equation}
	Notice that the eigenspace $E_{\mu}$ is finite-dimensional, we conclude the proof of \eqref{equ:operatornorm}. 
\end{proof}

Building upon the preceding theorem, we can establish the subsequent spectral approximation result:
\begin{theorem}\label{thm:eigenvalueerror}
Let $\mu$ be an eigenvalue of $T$ with multiplicity $m$, and $\mu^{p,h}_i$ ($i=1, \cdots, m$) be the corresponding discrete eigenvalues. Let $E_{\mu}$ denote the corresponding eigenvalue space, and $E_{\mu}^{p,h}$ be the corresponding discrete eigenvalue space. Suppose $E_{\mu} \subset H^{k}(\Omega_+\cup \Omega_-)$ with $k\ge p+1$. Then, we have
\begin{align}\label{equ:eigenvalueerror}
	|\mu - \mu_i^{p,h}| \lesssim \frac{h^{2\min(p+1,k)-2}}{p^{2k-3}}, \quad 1 \leq i \leq m.
\end{align}
For any $u \in E_{\mu}$, there exists $u^{p,h} \in E_{\mu}^{p,h}$ such that 
\begin{equation}
	\| u - u^{p,h}\|_{0, \Omega} \le \frac{h^{\min(p+1,k)}}{p^{(k-1)}} \|u\|_{k, \Omega_+\cup \Omega_-}.
\end{equation}
\end{theorem}
\begin{proof}
Let $u_1, \cdots, u_m$ be an orthonormal basis of $E_{\mu}$. By Theorem 7.3 in \cite{BaOs1991}, we have 
\begin{equation}
	\left|\mu - \mu_i^{p,h}\right| \lesssim \sum_{j, \ell =1}^m\left|\left((T-T^{p,h}) u_j, u_{\ell}\right)\right|+\left\|\left(T-T^{p,h}\right)\rvert_{E_\mu}\right\|_{\mathcal{L}(L^2(\Omega))}.
\end{equation}
Since the bound for the second term in the above inequality has been already established in Theorem \ref{thm:opapp}, it suffices to estimate the first term.  Using the triangle inequality, we have 
\begin{equation}
	((T-T^{p,h}) u_j, u_{\ell})\lesssim
		((T-\tilde{T}^{p,h}) u_j, u_{\ell}) + ((\tilde{T}^{p,h}-T^{p,h}) u_j, u_{\ell}) \coloneqq I_1 + I_2.
\end{equation}
For $I_1$, we have
\begin{equation}
	\begin{aligned}
		I_1 =  &\left(u_{\ell}, (T-\tilde{T}^{p,h}) u_j\right) \\
		     = & a_{p,h}(Tu_{\ell},  (T-\tilde{T}^{p,h}) u_j) \\
		     = & a_{p,h}((T-\tilde{T}^{p,h})u_{\ell},  (T-\tilde{T}^{p,h}) u_j) +  a_{p,h}(\tilde{T}^{p,h}u_{\ell},  (T-\tilde{T}^{p,h}) u_j)\\
		      = & a_{p,h}((T-\tilde{T}^{p,h})u_{\ell},  (T-\tilde{T}^{p,h}) u_j) +  a_{p,h}((T-\tilde{T}^{p,h}) u_j,\tilde{T}^{p,h}u_{\ell})\\
		      \lesssim & \vertiii{(T-\tilde{T}^{p,h})u_{\ell}}\vertiii{(T-\tilde{T}^{p,h})u_{j}}
		       + \frac{1}{h^2}g_{p,h} (\tilde{T}^{p,h} u_j,\tilde{T}^{p,h}u_{\ell}) \\
		       \lesssim & \vertiii{(T-\tilde{T}^{p,h})u_{\ell}}\vertiii{(T-\tilde{T}^{p,h})u_{j}} + \\
		       & \frac{1}{h^2}g_{p,h} (\tilde{T}^{p,h} u_j,\tilde{T}^{p,h}u_{j})^{1/2}g_{p,h} (\tilde{T}^{p,h} u_{\ell},\tilde{T}^{p,h}u_{\ell})^{1/2}\\
		       \lesssim &\frac{h^{2\min(p+1,k)-2}}{p^{(2k-3)}} \|u_j\|_{k, \Omega_+\cup \Omega_-}\|u_{\ell}\|_{k, \Omega_+\cup \Omega_-},
	\end{aligned}
\end{equation}
where we have used the energy error estimate in Theorem \ref{thm:energyerr} and adopted the Cauchy-Schwartz inequality and the same technique as in \eqref{equ:other} to estimate the ghost penalty term. 

For $I_2$, we have
\begin{equation}
	\begin{aligned}
		I_2 = & M_{p,h}(u_{\ell}, (\tilde{T}^{p,h}-T^{p,h}) u_j) \\
		 = & A_{p,h}(T^{p,h}u_{\ell}, (\tilde{T}^{p,h}-T^{p,h}) u_j) - g_{p,h}(u_{\ell}, (\tilde{T}^{p,h}-T^{p,h}) u_j)\\
		 = & A_{p,h}((T^{p,h}-\tilde{T}^{p,h})u_{\ell}, (\tilde{T}^{p,h}-T^{p,h}) u_j) + A_{p,h}(\tilde{T}^{p,h}u_{\ell}, (\tilde{T}^{p,h}-T^{p,h}) u_j)  \\
		 & - g_{p,h}(u_{\ell}, (\tilde{T}^{p,h}-T^{p,h}) u_j)\\
		 = & A_{p,h}((T^{p,h}-\tilde{T}^{p,h})u_{\ell}, (\tilde{T}^{p,h}-T^{p,h}) u_j) + g_{p,h}(u_j,\tilde{T}^{p,h}u_{\ell}) \\
		 & - g_{p,h}(u_{\ell}, (\tilde{T}^{p,h}-T^{p,h}) u_j)\\
		 \coloneqq & F_1 + F_2 + F_3. 
	\end{aligned}
\end{equation}
Using \eqref{equ:energydiff}, we have 
\begin{equation}
	F_1 \lesssim \frac{h^{2\min(p+1,k)-2}}{p^{2(k-1)}} \|u_j\|_{k, \Omega_+\cup \Omega_-}\|u_{\ell}\|_{k, \Omega_+\cup \Omega_-}.
\end{equation}
To estimate $F_2$, we have
\begin{equation}
	\begin{aligned}
		F_2 = & g_{p,h}(u_j - I^{p,h}u_j, \tilde{T}^{p,h}u_{\ell}) + g_{p,h}(I^{p,h}u_j, \tilde{T}^{p,h}u_{\ell}) \\
		 = & \left[ g_{p,h}(u_j - I^{p,h}u_j,u_j - I^{p,h}u_j)^{1/2} + g_{p,h}(I^{p,h}u_j, I^{p,h}u_j)^{1/2} \right]\\
		 &\times g_{p,h}(\tilde{T}^{p,h}u_{\ell}, \tilde{T}^{p,h}u_{\ell})^{1/2} \\
		  \lesssim &\frac{h^{2\min(p+1,k)-2}}{p^{2k-3}} \|u_j\|_{k, \Omega_+\cup \Omega_-}\|u_{\ell}\|_{k, \Omega_+\cup \Omega_-},
	\end{aligned}
\end{equation}
where we have used Theorem \ref{thm:ghostapp} and \eqref{equ:other}. Similarly, we can show 
\begin{equation}
	F_3 \le \frac{h^{2\min(p+1,k)-2}}{p^{2(k-1)}} \|u_j\|_{k, \Omega_+\cup \Omega_-}\|u_{\ell}\|_{k, \Omega_+\cup \Omega_-}.
\end{equation}
Combining all the above estimates completes the proof of \eqref{equ:eigenvalueerror}. 

Using Theorem 7.4 in \cite{BaOs1991}, we have 
\begin{equation}
\| u - u^{p,h}\|_{0, \Omega} \lesssim  \|(T-T^{p,h})|_{E_\mu}\|_{\mathcal{L}(L^2(\Omega))}\lesssim \frac{h^{\min(p+1,k)}}{p^{(k-1)}} \|u\|_{k, \Omega_+\cup \Omega_-}.
\end{equation}
\end{proof}

Using the relationship between $\mu$ ($\mu^{p,h}$) and $\lambda$ ($\lambda^{p,h}$) brings 
\begin{equation}\label{equ:eigenvalueerror2}
	|\lambda - \lambda_i^{p,h}| \lesssim \frac{h^{2\min(p+1,k)-2}}{p^{2k-3}}, \quad 1 \leq i \leq m.
\end{equation}

\section{Numerical experiments}
\label{sec:num}
In this section, we will provide a series of numerical examples to both substantiate our theoretical findings and showcase the improved robustness achieved through the inclusion of ghost penalty terms. For the first two examples, our computational domain is denoted as $\Omega = (-1, 1)\times (-1, 1)$. We generate a uniform partition $\mathcal{T}_h$ by subdividing $\Omega$ into $N^2$ sub-squares. The stabilizing constants must  be set. The interface eigenvalue problem requires both $\gamma_A,\gamma_M$ while the interface problem only necessitates the former. We set the values empirically, testing out different orders of magnitude for optimal results. For all $p$-convergence plots, we show the rate $O(p^{-8})$ simply to clearly highlight the difference between polynomial and the expected numerical spectral trends.


\subsection{Interface problems}
In this subsection, we will present two numerical examples to support the theoretical findings pertaining to elliptic interface problems. 

\subsubsection{Circular interface problem} 
In this example, we consider the elliptic interface problem \eqref{equ:bvp} with a circular interface of radius $r_0 = 0.5$. The exact solution is given by:
\[
u(r,\theta) = \begin{cases}
\frac{r^3}{\alpha_-}, & (r,\theta) \in \Omega_-, \\
\frac{r^3}{\alpha_+} + \left( \frac{1}{\alpha_-} - \frac{1}{\alpha_+} \right) r_0^3, & (r,\theta) \in \Omega_+,
\end{cases}
\]
where $r = \sqrt{x^2+y^2}$. We tested our numerical solutions for various choices of $\alpha_{\pm}$. For simplicity, we only present the numerical results for the case when $\alpha_+/\alpha_- = 1000$. The numerical results for other cases are similar.

  \begin{figure}[!h]
   \centering
  \subcaptionbox{\label{fig:circle_l2_p3}}
   {\includegraphics[width=0.31\textwidth]{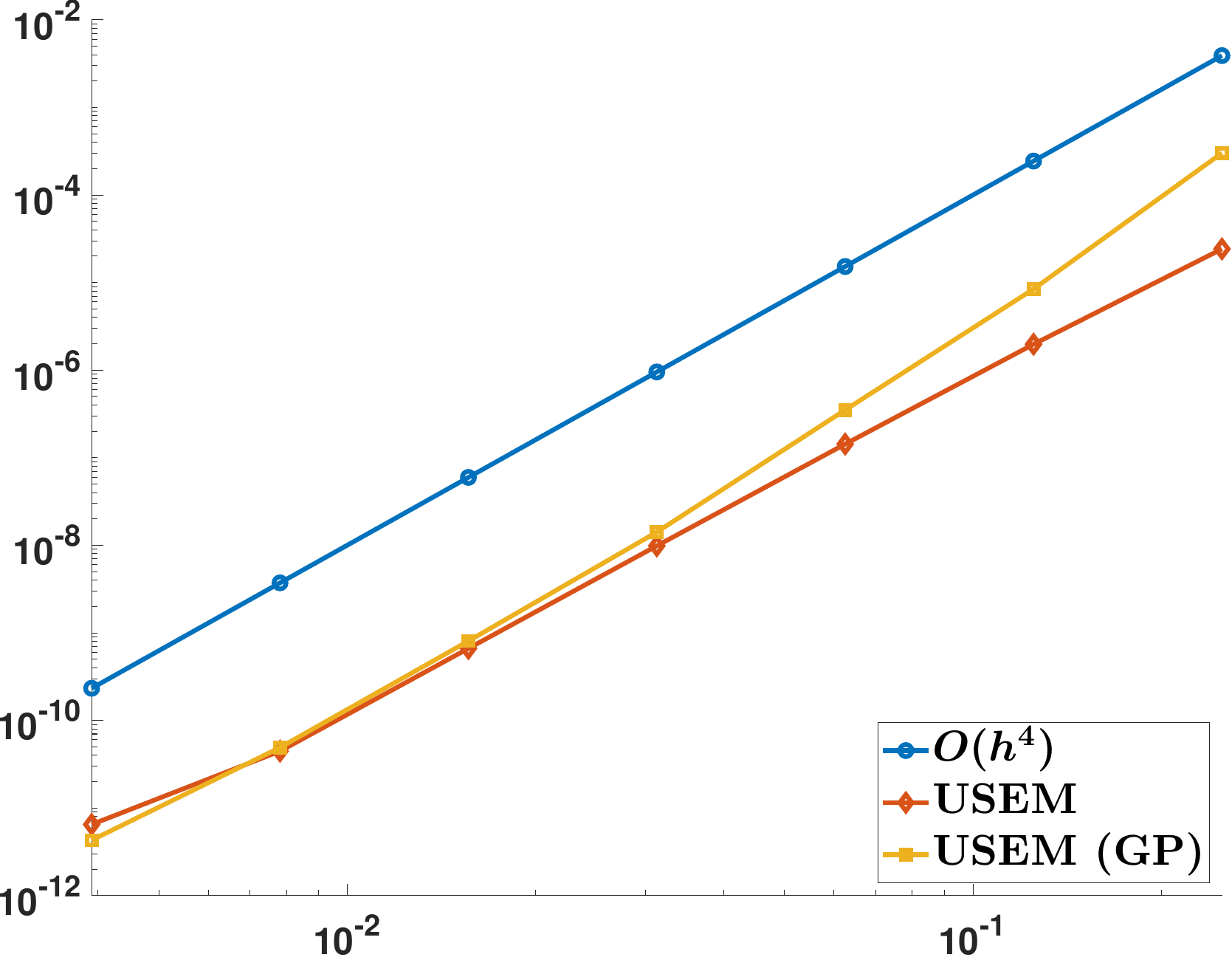}}
   \subcaptionbox{\label{fig:circle_h1_p3}}
   {\includegraphics[width=0.31\textwidth]{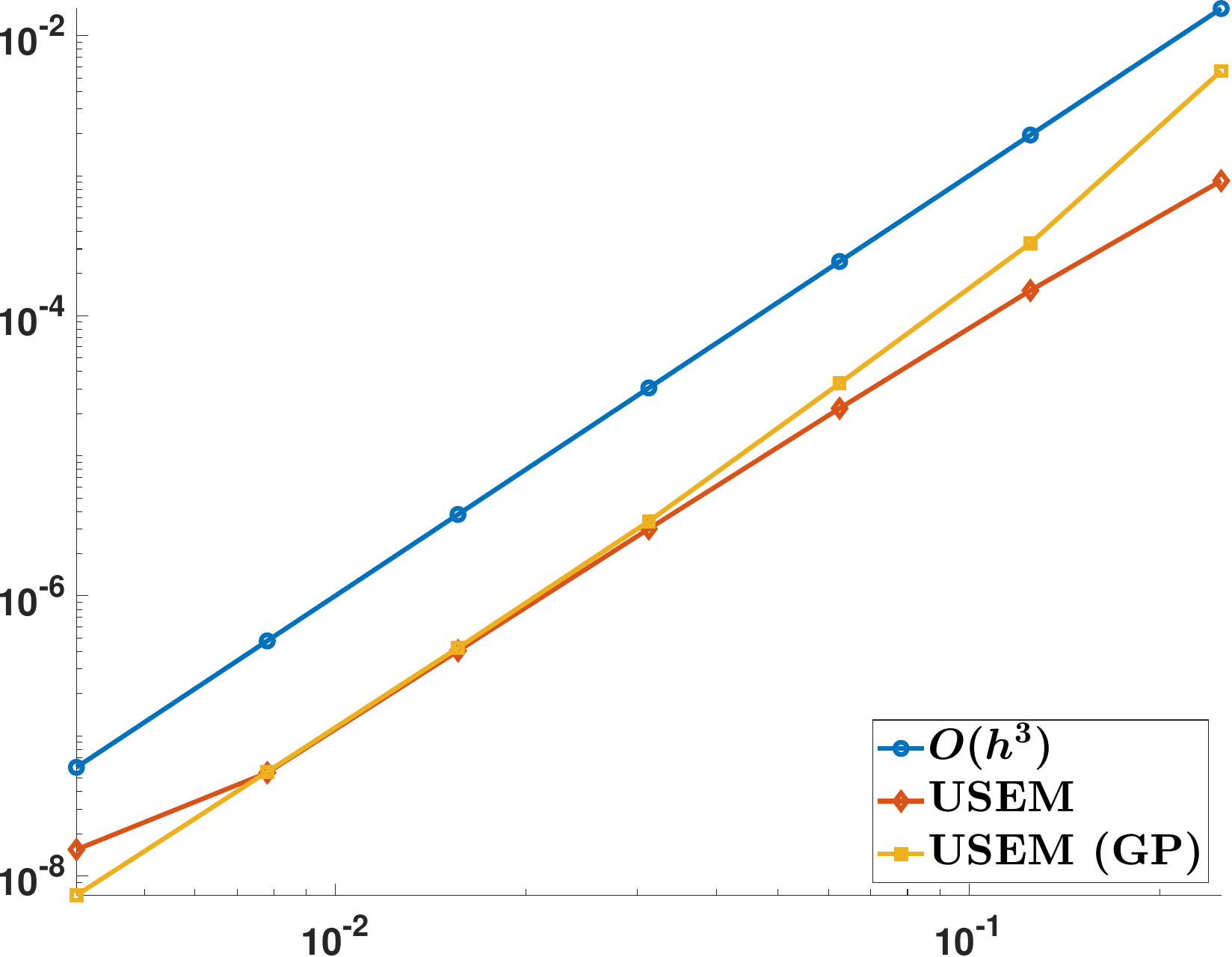}}
      \subcaptionbox{\label{fig:circle_cond_p3}}
   {\includegraphics[width=0.31\textwidth]{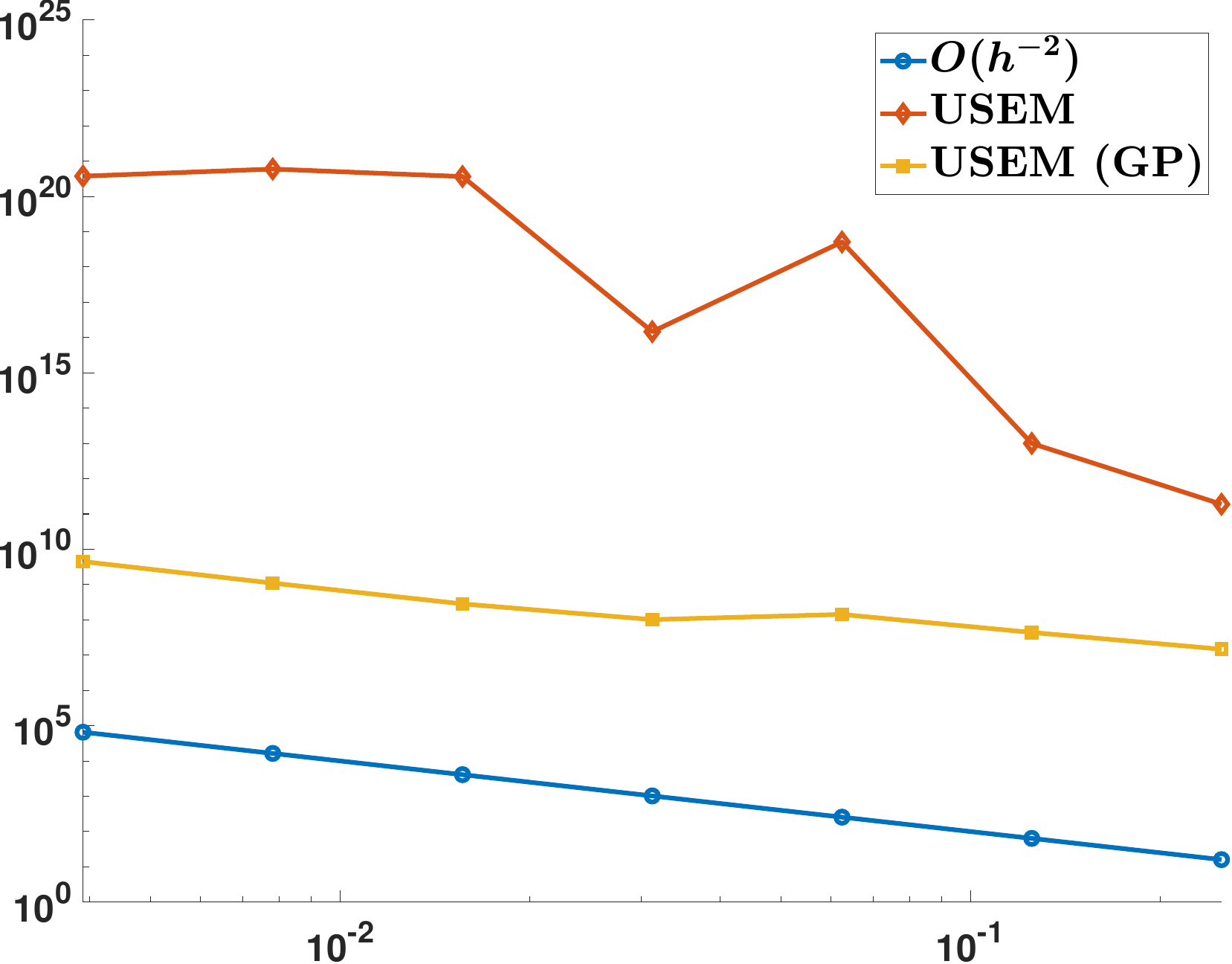}}
   \caption{Log-log $h$-convergence plots for circle interface problem with $\alpha_- = 1$ and $\alpha_+ = 1000$. Orange line: non-stabilized solution. Yellow Line: stabilized solution. (a)  $L^2$-error; (b) $H^1$-error; (c) Condition number. Stabilized with $\gamma_A = 0.1$.}
   \label{fig:circle_hconvergence}
\end{figure}

  \begin{figure}[!h]
   \centering
  \subcaptionbox{\label{fig:circle_l2_h16}}
   {\includegraphics[width=0.31\textwidth]{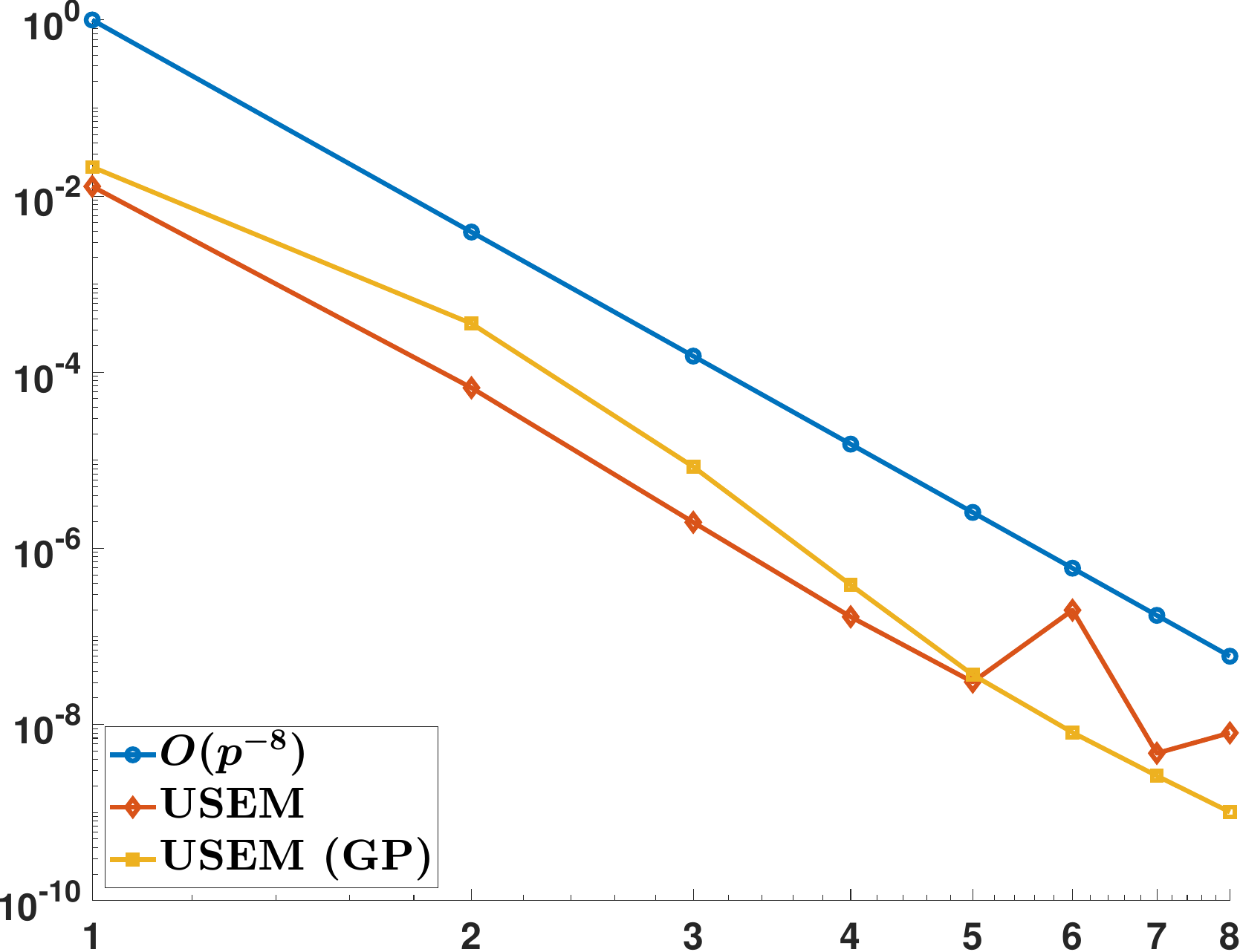}}
   \subcaptionbox{\label{fig:circle_h1_h16}}
   {\includegraphics[width=0.31\textwidth]{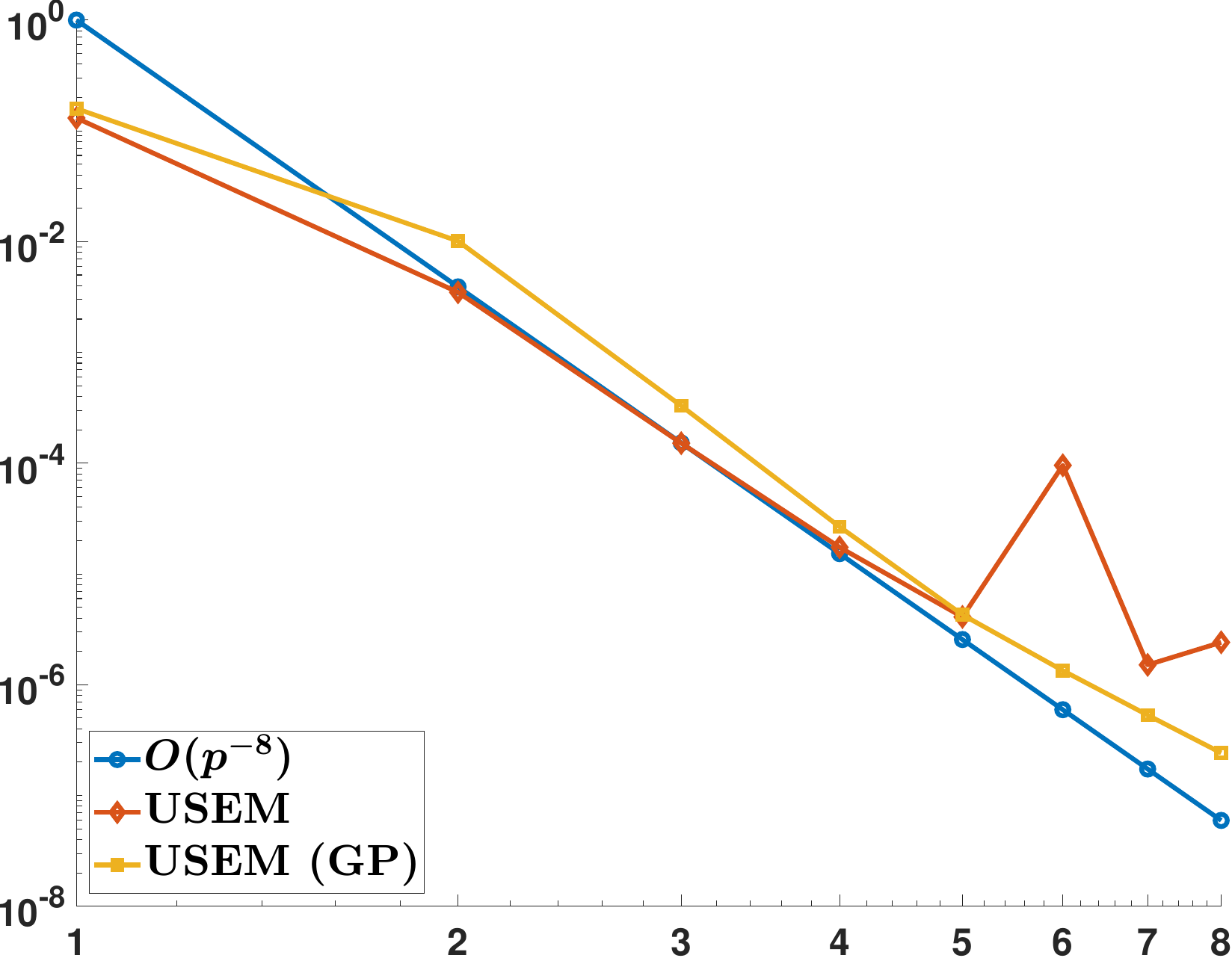}}
      \subcaptionbox{\label{fig:circle_cond_h16}}
   {\includegraphics[width=0.31\textwidth]{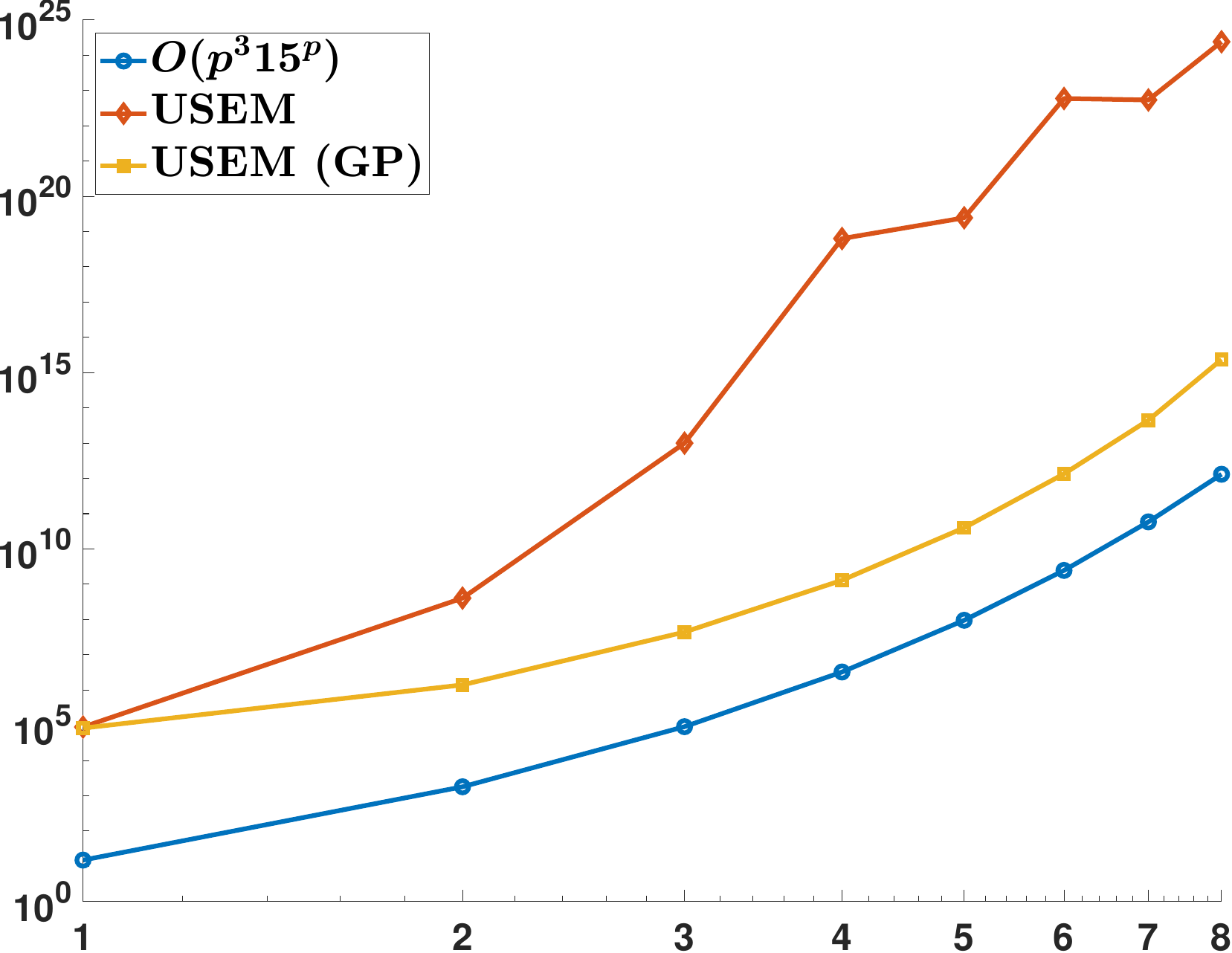}}
   \caption{Semi-log $p$-convergence plots for circle interface problem with $\alpha_- = 1$ and $\alpha_+ = 1000$. Orange line: non-stabilized solution. Yellow Line: stabilized solution. (a)  $L^2$-error; (b) $H^1$-error; (c) Condition number. Stabilized with $\gamma_A = 0.05$. }
   \label{fig:circle_pconvergence}
\end{figure}

Firstly, we present $h$-convergence results in Figure \ref{fig:circle_hconvergence} for $p=3$. As expected from our prior analysis, we observe convergence rates following:
\begin{equation*}
    O(h^{\min(p+1,m)}) = O(h^{p+1}) = O(h^4),
\end{equation*}
in the $L^2$-norm, which is consistent with our theoretical results. Similarly, a convergence rate of $O(h^p)$ can be observed for the $H^1$ error. We observe that our stabilized ghost penalty (GP) version outperforms the standard approach in all aspects, preserving convergence rates in a superior fashion. Our graphs even suggest that this difference would become more pronounced as $h$ decreases further. Finally, we appreciate a tremendous improvement in the stiffness condition number with ghost penalty stabilization, with its evolution resembling $O(h^{-2})$ growth, similar to fitted Finite Element Methods.

Figure \ref{fig:circle_pconvergence} presents $p$-convergence results for a grid of $16 \times 16$ square elements. Regarding the $L^2$-error, only the ghost penalty (GP) version noticeably exhibits spectral convergence, initially deviating from the reference polynomial line (blue). In contrast, the non-stabilized version shows no apparent curvature, indicating a lack of spectral behavior. It's worth noting that numerical limitations arising from double-precision arithmetic may hinder further spectral behavior beyond a polynomial degree of six. The ghost penalty stabilization maintains a cleaner trajectory as we approach the degree eight limit, while the standard USEM approach becomes erratic and yields somewhat unreliable results.
Similar properties extend to the $H^1$-norm, although the overall convergence quality is notably reduced. We can only discern a hint of spectral behavior for low-order polynomial  bases in the GP case. Once again, a significant difference exists between the algorithms in terms of the evolution of the condition number, with even the ghost penalty version growing beyond an exponential rate.

\subsubsection{Flower shape interface problem}
In this example, we consider a flower shaped interface problem. The interface curve $\Gamma$ in polar coordinates is given by
\[
r = \frac{1}{2} + \frac{\sin(5\theta)}{7},
\]
which contains both convex and concave parts. The diffusion coefficient is piecewise constant with $\alpha_{-} = 1$ and $\alpha_{+} = 10$. The right-hand function $f$ in \eqref{equ:bvp} is chosen to match the exact solution
\[
u(r,\theta) = \begin{cases}
    e^{r^2} & \text{if } (r,\theta) \in \Omega^{-}, \\
    0.1r^4 - 0.01 \ln(2r) & \text{if } (r,\theta) \in \Omega^{+}.
\end{cases}
\]
In this case, both interface jump conditions are nonhomogeneous and computed through the exact solution.

  \begin{figure}[!h]
   \centering
  \subcaptionbox{\label{fig:flower_l2_p3}}
   {\includegraphics[width=0.31\textwidth]{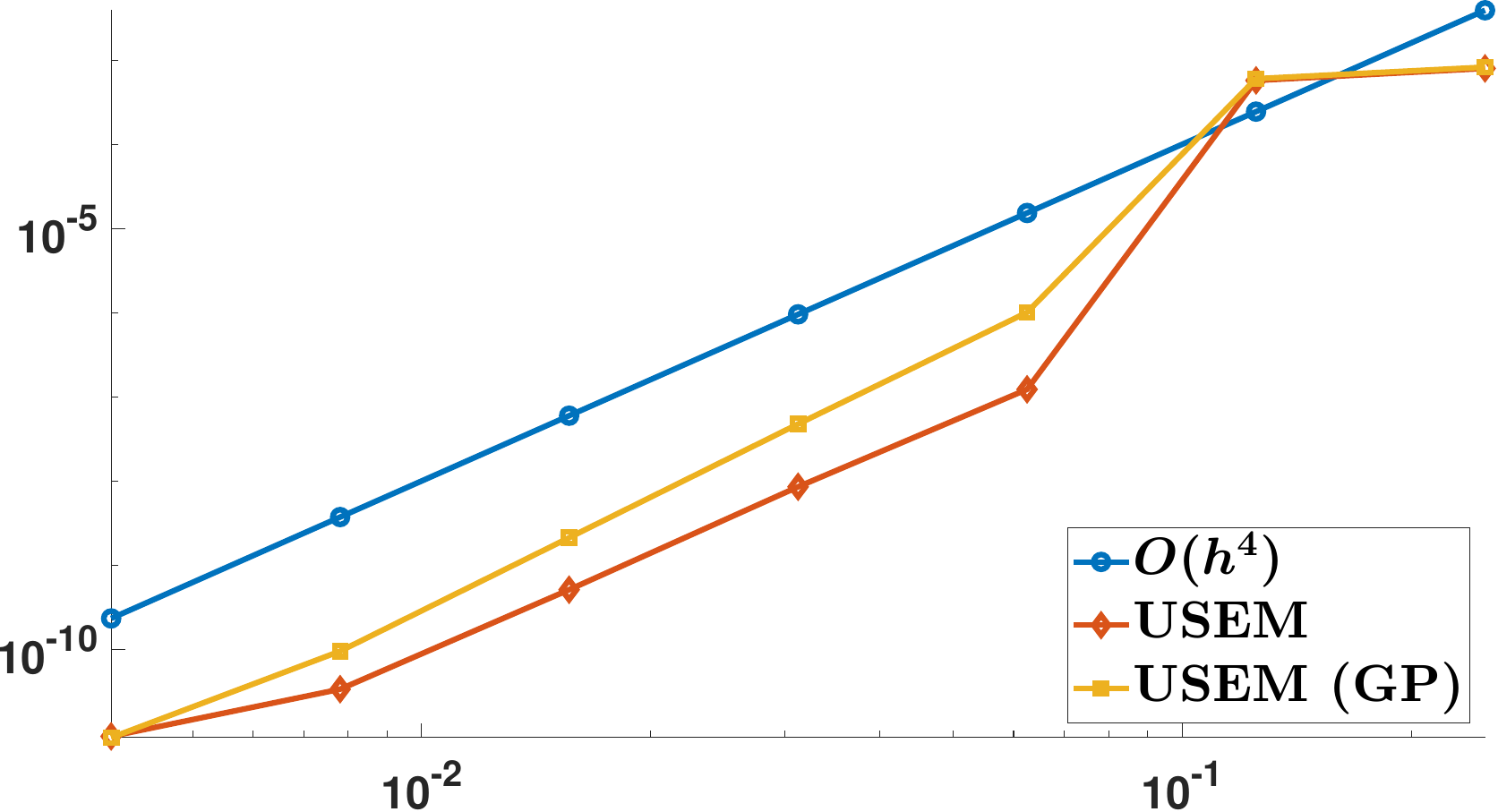}}
   \subcaptionbox{\label{fig:flower_h1_p3}}
   {\includegraphics[width=0.31\textwidth]{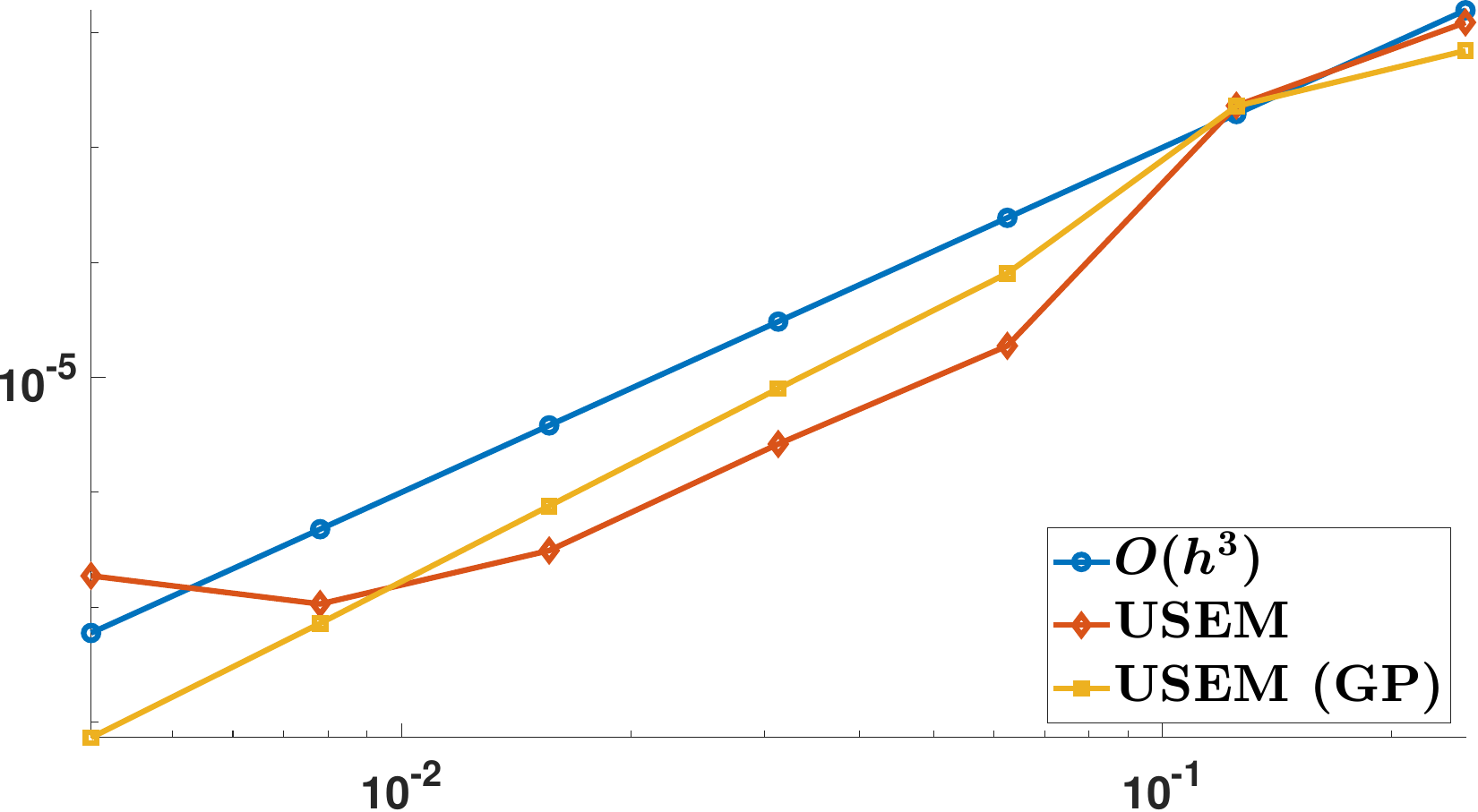}}
      \subcaptionbox{\label{fig:flower_cond_p3}}
   {\includegraphics[width=0.31\textwidth]{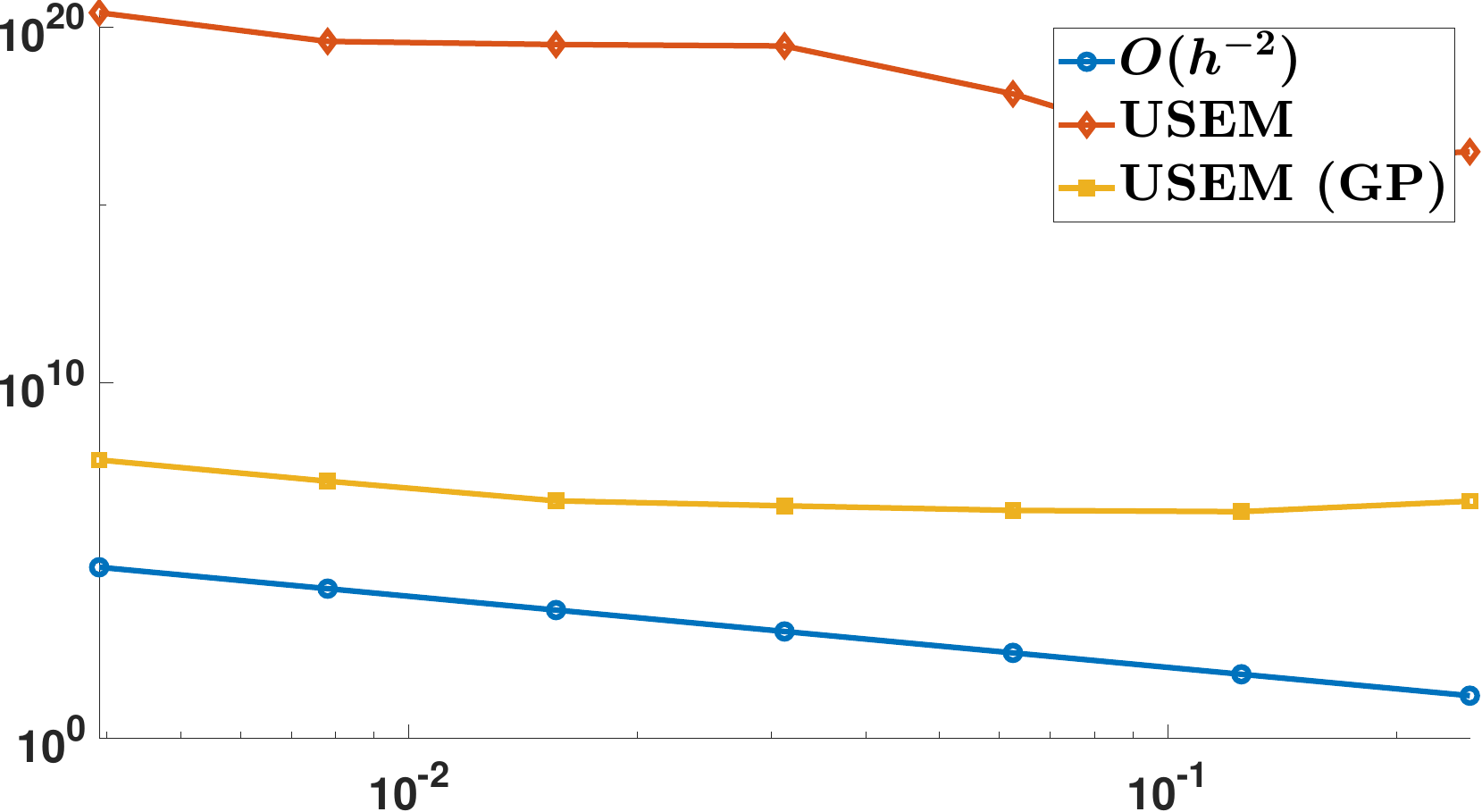}}
   \caption{Log-log $h$-convergence plots for flower shape interface problem with $\alpha_- = 1$ and $\alpha_+ = 10$. Orange line: non-stabilized solution. Yellow Line: stabilized solution. (a) $L^2$-error; (b) $H^1$-error; (c) Condition number. Stabilized with $\gamma_A = 0.001$.}
   \label{fig:flower_hconvergence}
\end{figure}

  \begin{figure}[!h]
   \centering
  \subcaptionbox{\label{fig:flower_l2_h16}}
   {\includegraphics[width=0.31\textwidth]{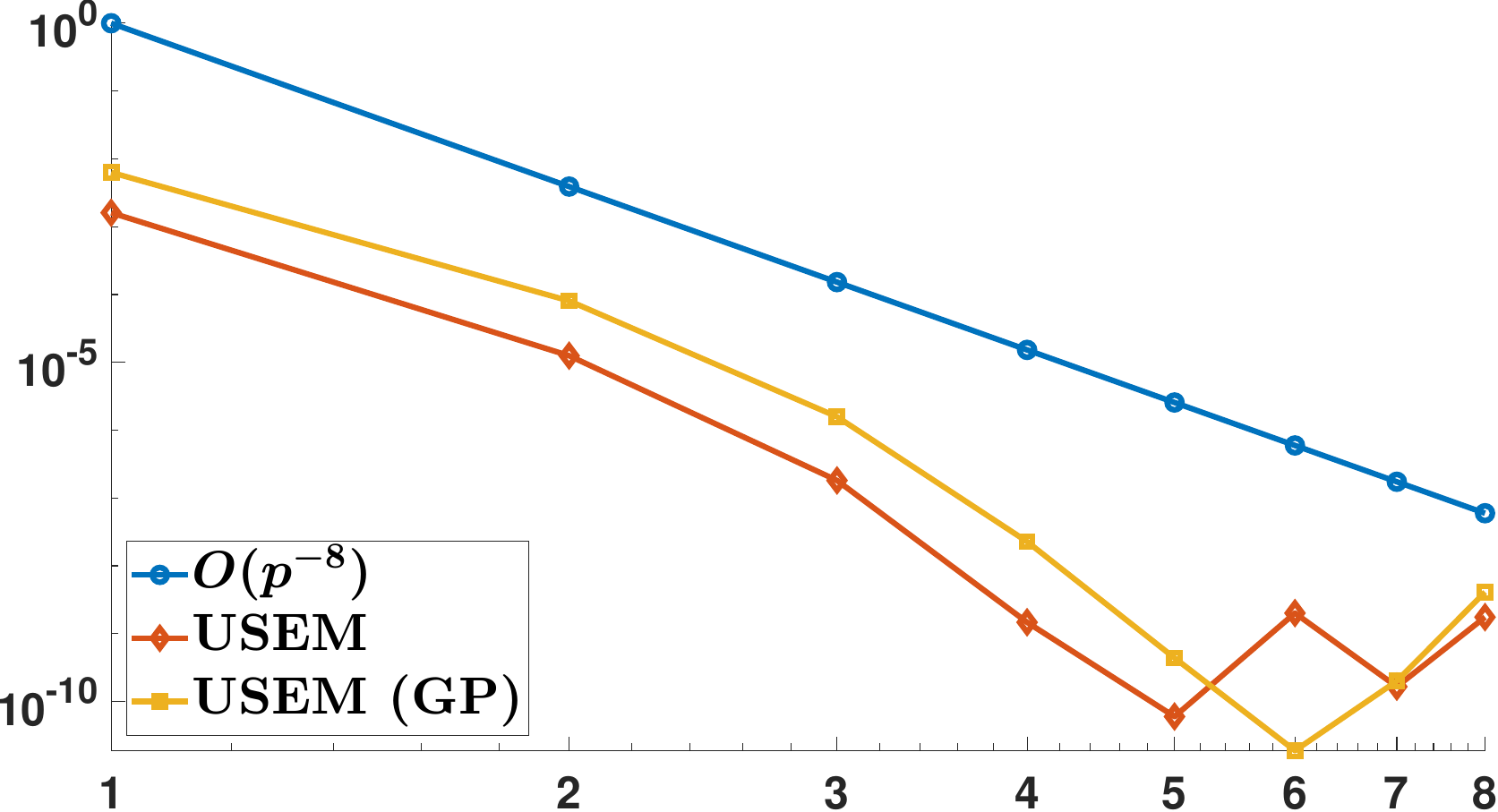}}
   \subcaptionbox{\label{fig:flower_h1_h16}}
   {\includegraphics[width=0.31\textwidth]{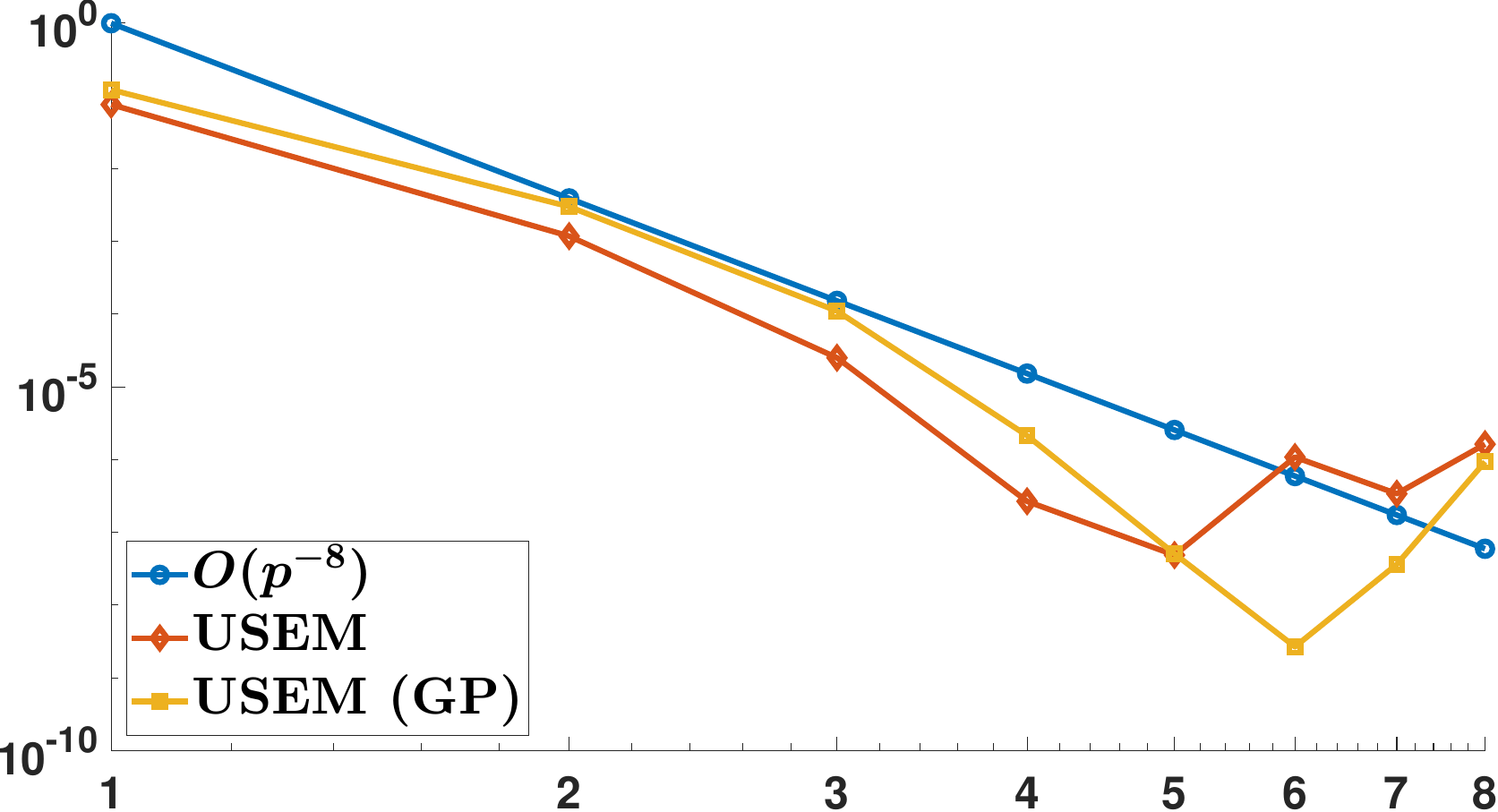}}
      \subcaptionbox{\label{fig:flower_cond_h16}}
   {\includegraphics[width=0.31\textwidth]{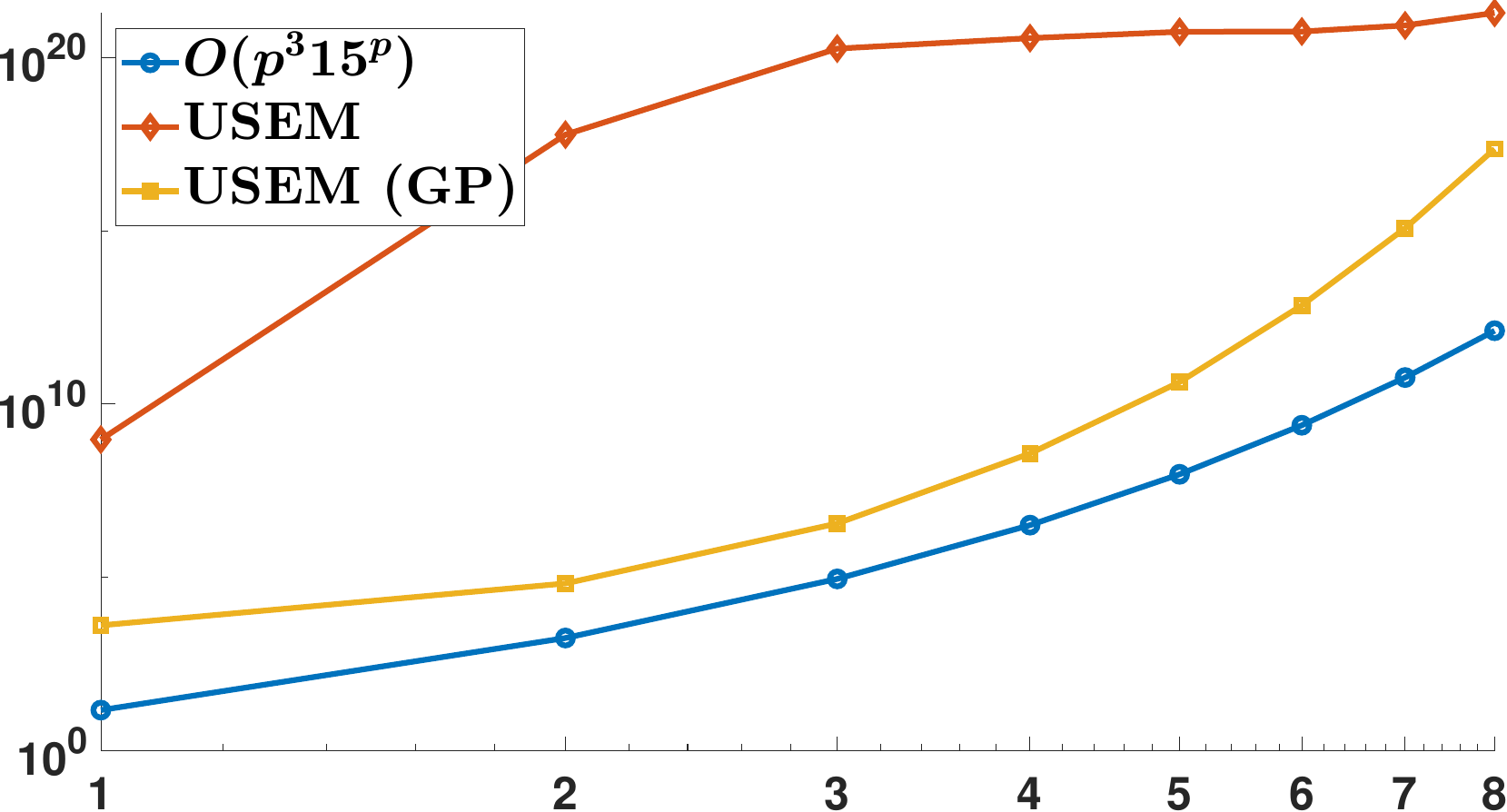}}
   \caption{Semi-log $p$-convergence plots for flower shape interface problem with $\alpha_- = 1$ and $\alpha_+ = 10$. Orange line: non-stabilized solution. Yellow Line: stabilized solution. (a) $L^2$-error; (b) $H^1$-error; (c) Condition number. Stabilized with $\gamma_A = 0.005$.}
   \label{fig:flower_pconvergence}
\end{figure}

The $h$-convergence results presented in Figure \ref{fig:flower_hconvergence} align with our theoretical expectations. In these simulations, we continued to use third-order ($p=3$) basis functions in our  calculations. It is evident that the $L^2$ convergence rate of the standard unfitted spectral element method deteriorates as we reach the last data point (the smallest $h$ value). In contrast, the ghost penalty (GP) version exhibits superior stability, maintaining a more consistent convergence rate.
This trend is also observed in the $H^1$ error. As we refine the mesh, the standard unfitted spectral element method becomes increasingly susceptible to weaker numerical convergence rates, a characteristic not shared by its ghost penalty counterpart. In summary, the GP version consistently outperforms the standard unfitted spectral element method, and in some cases, even surpasses fitted finite element methods in terms of stiffness condition number growth. Notably, the yellow line in the right plot exhibits a less steep slope compared to the reference $O(h^{-2})$ blue line.

Given the more intricate geometry, we employed a finer grid, as compared to the previous circular test, consisting of $29 \times 29$ elements. This choice allows us to better observe the desired $p$-convergence rates. In Figure \ref{fig:flower_pconvergence}, spectral convergence becomes quite evident. Both our unfitted spectral element method (USEM) lines, with and without ghost penalty (GP) stabilization, curve away from the reference polynomial rate (blue) before eventually succumbing to the limitations imposed by numerical precision.
It's important to note that the ghost penalty terms involve high-order derivatives, which can contribute significantly to round-off errors at higher degrees. In this case, we can observe that spectral behavior is preserved from $L^2$ to $H^1$ norms. Although our USEM curves eventually become inconsistent, the spectral tendency persists for two more degrees of $p$ with GP stabilization, thereby numerically validating its effectiveness beyond the realm of $h$-convergence.
Lastly, it's worth noting that ghost penalty demonstrates a remarkable improvement in the progression of the stiffness condition number. However, we also observe a growth rate that exceeds the exponential rate.

\subsection{Interface eigenvalue problems}
In this example, we investigate the interface eigenvalue problem \eqref{equ:eigenprob}. Our computational domain is $\Omega = (0,\pi) \times (0,\pi)$, which contains a circular interface centered at $(\frac{\pi}{2},\frac{\pi}{2})$ with a radius of $\frac{\pi}{4}$, effectively splitting $\Omega$ into subdomains $\Omega_-$ and $\Omega_+$. Unlike previous cases, we now have to consider not only the stiffness matrix but also the mass matrix, adding another potential source of ill-conditioning for numerical solvers.

We do not have theoretical eigenvalues to compare with. We quantify the error in the following manner,
\begin{equation*}
    \varepsilon_k(\lambda_i) = \frac{|\lambda_{i,k+1}-\lambda_{i,k}|}{|\lambda_{i,k}|},
\end{equation*}
where $\lambda_{i,k}$ represents the eigenvalue $\lambda_i$ at the $k^{th}$ iteration. Iteration here can refer to either mesh refinement in $h$-convergence or polynomial degree in $p$-convergence.
Our plots show results for (\ref{equ:eigenprob}) five smallest eigenvalues $\{\lambda_i\}_{i=0}^4$, for both non and fully stabilized setups.

The ghost penalty's contribution is thus even more significant in this scenario, as it effectively addresses instability issues stemming from both stiffness and mass matrices. We only present the numerical results for the specific case of $\alpha_+/\alpha_- = 1000$ to demonstrate the efficacy of the ghost penalty method.

  \begin{figure}[!h]
   \centering
  \subcaptionbox{\label{fig:eigen_l2_p3}}
   {\includegraphics[width=0.47\textwidth]{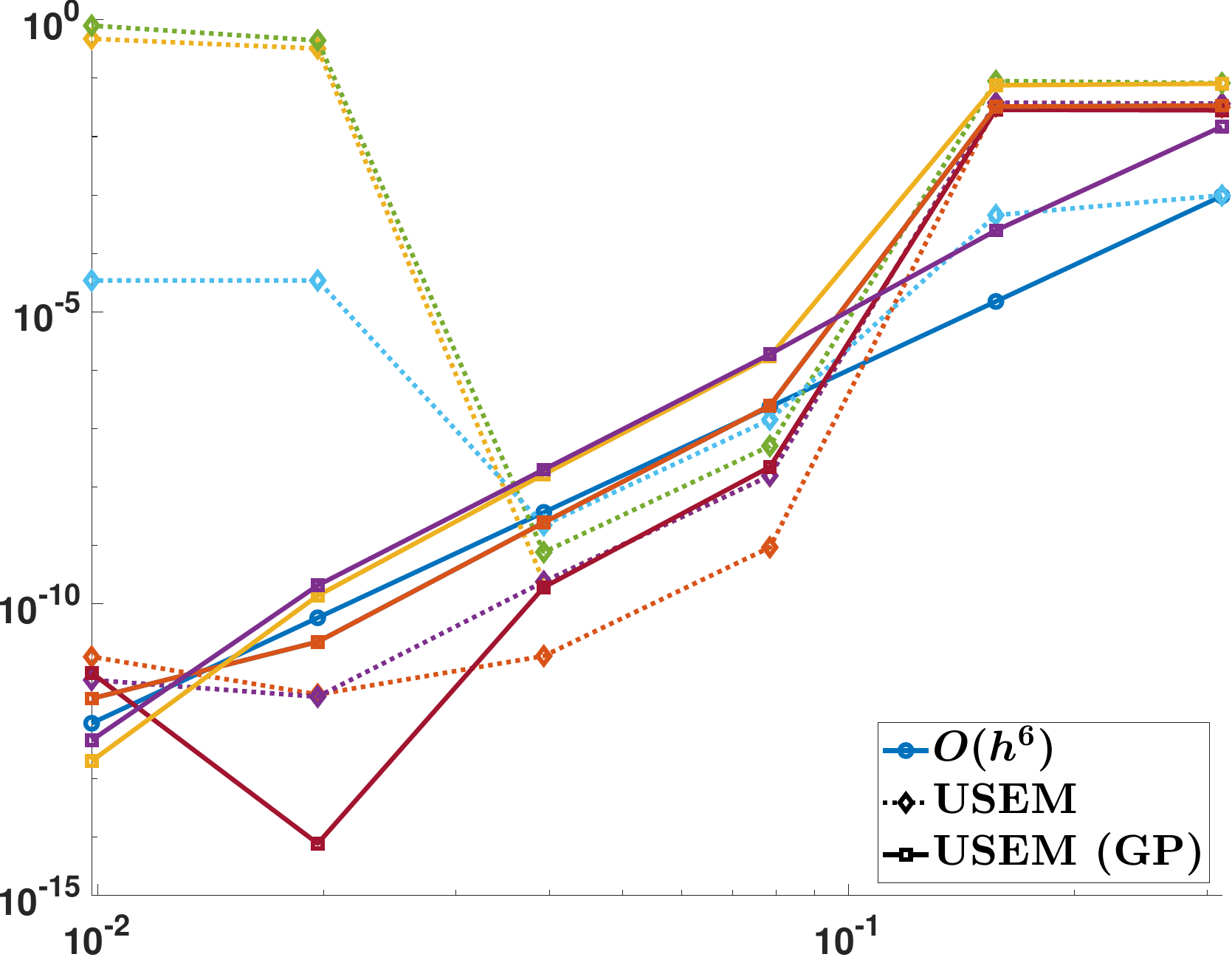}}
      \subcaptionbox{\label{fig:eigen_cond_p3}}
   {\includegraphics[width=0.47\textwidth]{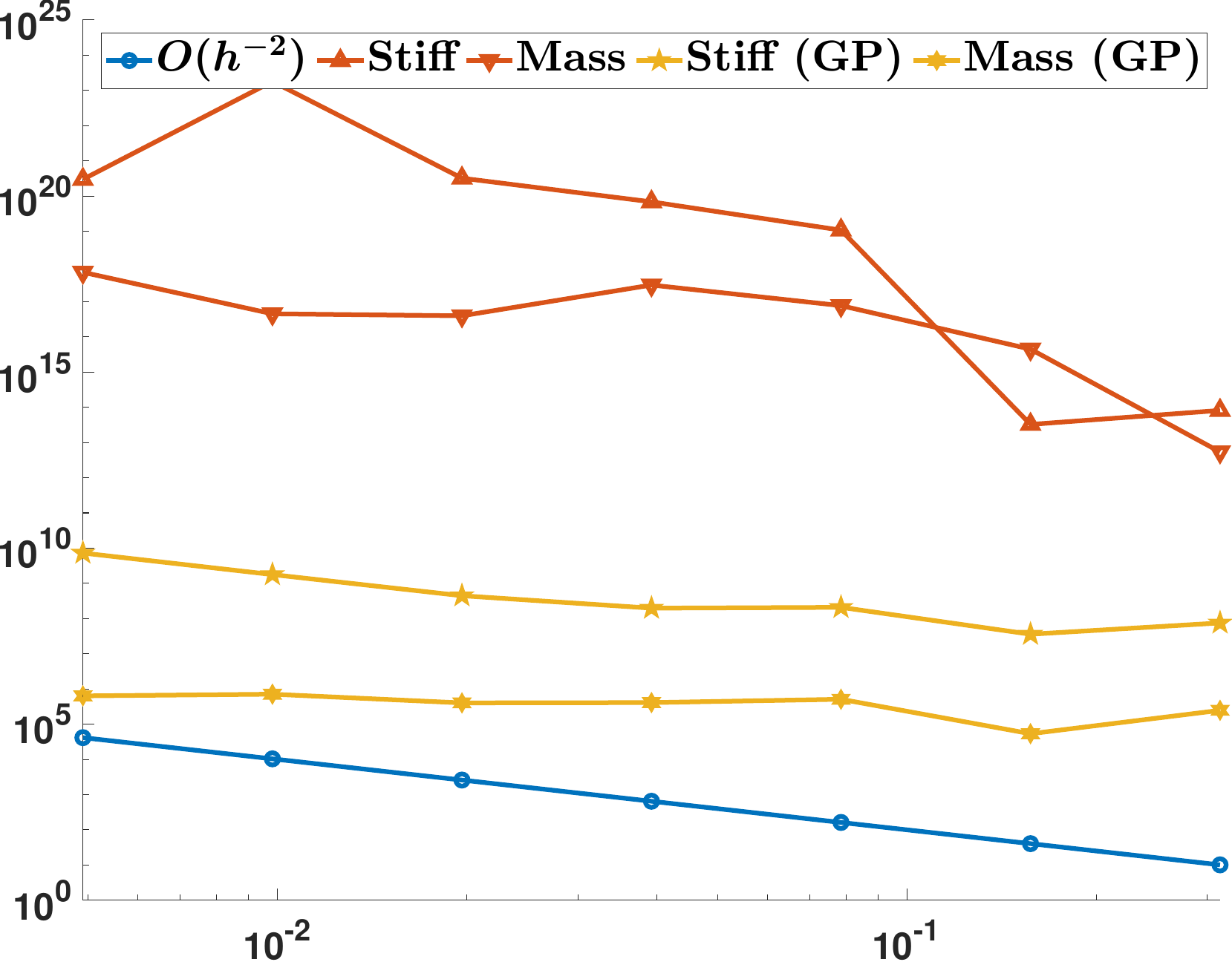}}
   \caption{Log-log $h$-convergence plots for circle interface eigenvalue problem with $\alpha_- = 1$ and $\alpha_+ = 1000$: (a) Eigenvalues $\{\lambda_i\}_{i=0}^4$. Dotted lines with diamond markers: non-stabilized problem. Solid lines with square markers: fully-stabilized problem. (b) Condition number. Orange solid lines with triangle markers: non-stabilized problem. Yellow solid lines with asterisk markers: fully-stabilized problem.}
   \label{fig:eigen_hconvergence}
\end{figure}

  \begin{figure}[!h]
   \centering
  \subcaptionbox{\label{fig:eigen_l2_h16}}
   {\includegraphics[width=0.47\textwidth]{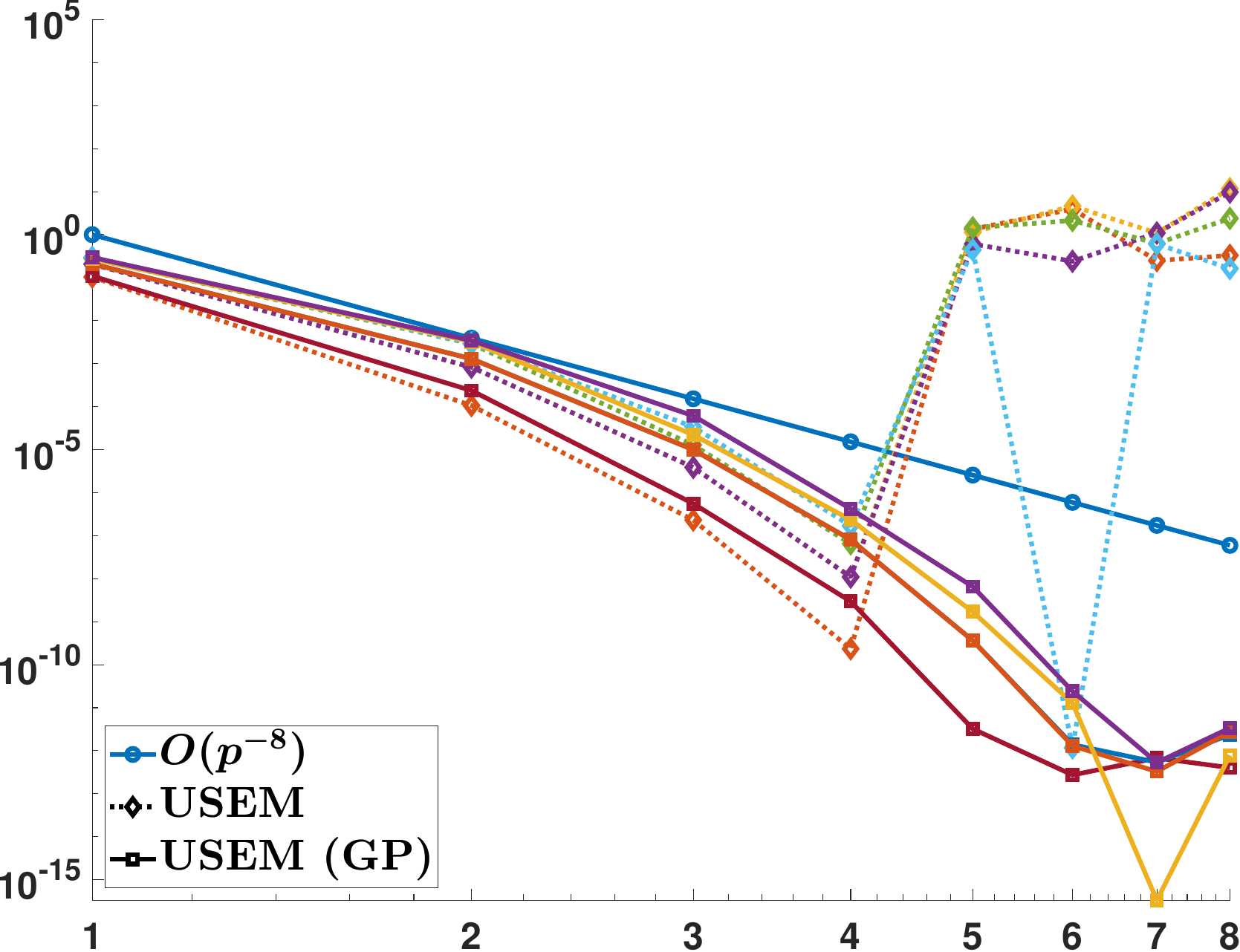}}
      \subcaptionbox{\label{fig:eigen_cond_h16}}
   {\includegraphics[width=0.47\textwidth]{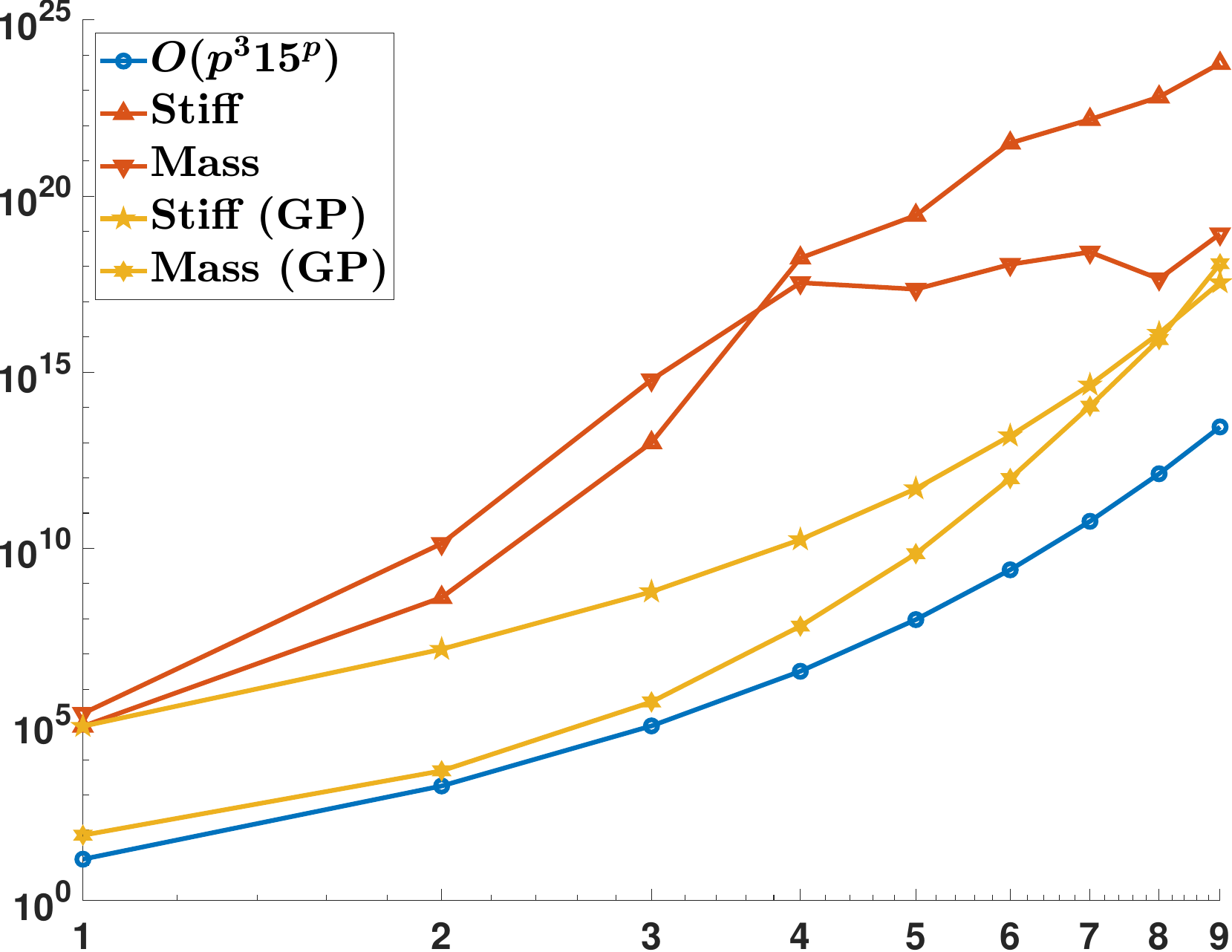}}
   \caption{Semi-log $p$-convergence plots for circle interface eigenvalue problem with $\alpha_- = 1$ and $\alpha_+ = 1000$. (a) Eigenvalues $\{\lambda_i\}_{i=0}^4$. Dotted lines with diamond markers: non-stabilized problem. Solid lines with square markers: fully-stabilized problem. (b) Condition number. Orange solid lines with triangle markers: non-stabilized problem. Yellow solid lines with asterisk markers: fully-stabilized problem.}
   \label{fig:eigen_pconvergence}
\end{figure}

In the context of $h$-convergence (Figure \ref{fig:eigen_hconvergence}), we set the stabilizing coefficients to $\gamma_A = 4.1$ and $\gamma_M = 0.002$ respectively. Aside from the initial step, we observe our expected rate of convergence,
$$ O(h^{2\min(p+1,m)-2}) = O(h^{2(p+1)-2}) = O(h^{2p}) = O(h^6),$$
for the USEM with ghost penalty Stabilization. This convergence rate aligns with our theoretical analysis and demonstrates the effectiveness of the ghost penalty method.
In contrast, standard USEM struggles to maintain the theoretical convergence trajectory and becomes unstable and unreliable as we refine the mesh and reduce the mesh-to-interface intersections. Additionally, there is a noticeable difference in the progression of matrix condition numbers between the two methods. From the outset, the non-ghost penalty version is already in or close to the ill-conditioned range, while the ghost penalty matrices, both stiffness and mass, exhibit convergence rates of approximately $O(h^{-2})$ and $O(1)$ respectively, which are in line with their fitted counterparts.

For $p$-convergence in Figure \ref{fig:eigen_pconvergence}, we set $\gamma_A = 0.1$ and $\gamma_M = 0.05$. In this case, both versions of USEM initially exhibit spectral convergence, with their curves deviating significantly from the reference polynomial line (blue). This behavior is notably different from the Poisson problem with the same domains $\Omega$ and $\Gamma$.
However, the non-stabilized method breaks down after reaching degree $p=4$, with the error curve diverging and exhibiting random oscillations. On the other hand, GP USEM continues to descend in a spectral fashion, approaching machine precision (double arithmetic) and remaining stable even at higher degrees.
Regarding the condition numbers, the main observation is that ghost penalty delays the inevitable and rapid increase in condition numbers, which is one of the factors contributing to the improved numerical stability of the GP algorithm. Interestingly, the stabilized mass matrix's condition number eventually reaches and appears to overtake that of the stiffness matrix. This observation is somewhat surprising since the mass bilinear form for is much simpler than for stiffness.

\section{Conclusion}
\label{sec:con}
In this paper, we have introduced a novel spectral element method on unfitted meshes. Our proposed method combines the spectral accuracy of spectral element methods with the geometric flexibility of unfitted Nitsche's methods. To enhance the robustness of our approach, especially for small cut elements, we have introduced a tailored ghost penalty term with a polynomial degree of $p$. We have demonstrated the optimal $hp$ convergence properties of our proposed methods.
We have conducted extensive numerical experiments to validate our theoretical results. These numerical examples not only confirm the $h$-convergence observed in existing literature but also showcase the $p$-convergence of our method.


\section*{Acknowledgment}
H.G. acknowledges partial support from the Andrew Sisson Fund, Dyason Fellowship, and the Faculty Science Researcher Development Grant at the University of Melbourne. X.Y. acknowledges partial support from the NSF grant DMS-2109116. H.G. would like to express gratitude to Prof. Jiayu Han from Guizhou Normal University for valuable discussions on eigenvalue approximation.

%

\bibliographystyle{siamplain}
\bibliography{References}
\end{document}